\documentclass[11pt]{article}
\usepackage{makeidx}
\usepackage{amssymb}
\usepackage{amsfonts}
\usepackage{graphicx}
\usepackage{epstopdf}
\usepackage{amsmath}
\usepackage{amsthm}
\usepackage{nicefrac}
\usepackage{floatrow}
\usepackage{url}
\usepackage{here}
\usepackage{bbm}
\usepackage{dsfont}
\usepackage{algorithm}
\usepackage{algpseudocode}

\newtheorem{theorem}{Theorem}

\newtheorem{corollary}[theorem]{Corollary}

\newtheorem{proposition}[theorem]{Proposition}
\theoremstyle{definition}
\newtheorem{remark}[theorem]{Remark}

\newcommand{\tdiscr}{\tau_{\mathrm{discr}}}
\newcommand{\tres}{\tau_{\mathrm{res}}}
\begin{document}

\title{Embedded techniques for choosing the parameter \\
in Tikhonov regularization}
\author{S. Gazzola, P. Novati, M.~R. Russo\\
Department of Mathematics\\
University of Padua, Italy}
\maketitle

\begin{abstract}
\noindent This paper introduces a new strategy for setting the
regularization parameter when solving large-scale discrete ill-posed linear
problems by means of the Arnoldi-Tikhonov method. This new rule is
essentially based on the discrepancy principle, although no initial
knowledge of the norm of the error that affects the right-hand side is
assumed; an increasingly more accurate approximation of this quantity is
recovered during the Arnoldi algorithm. Some theoretical estimates are
derived in order to motivate our approach. Many numerical experiments,
performed on classical test problems as well as image deblurring are
presented.\newline
\end{abstract}

\section{Introduction}

Let us consider a linear discrete ill-posed problem of the form
\begin{equation}
Ax=b,  \label{sys}
\end{equation}%
where $A\in \mathbb{R}^{N\times N}$ is severely ill-conditioned and may be
of huge size. These sort of systems typically arise from the discretization
of Fredholm integral equations of the first kind with compact kernel (for an
exhaustive background on these class of problems, cf. \cite[Chapter 1]{PCH}%
). The right-hand side $b$ is assumed to be affected by an unknown additive
error $e$ coming from the discretization process or measurements
inaccuracies, i.e.,
\begin{equation}
b=b^{ex}+e,  \label{available}
\end{equation}%
where $b^{ex}$ denotes the unknown exact right-hand side. We assume that the
unperturbed system $Ax=b^{ex}$ is consistent and we denote its solution by $%
x^{ex}$; the system (\ref{sys}) is not guaranteed to be consistent.
Referring to the Singular Value Decomposition (SVD) of the matrix $A$,
\begin{equation}
A=U\Sigma V^{T},  \label{SVD}
\end{equation}%
we furthermore assume that the singular values $\sigma _{i}$ quickly decay
toward zero with no evident gap between two consecutive ones.

Because of the ill-conditioning of $A$ and the presence of noise in $b$, in
order to find a meaningful approximation of $x^{ex}$ we have to substitute
the available system (\ref{sys}) with a nearby problem having better
numerical properties: this process is called regularization. One of the most
well-known and well-established regularization technique is Tikhonov method
that, in its most general form, can be written as
\begin{equation}
\min_{x\in \mathbb{R}^{N}}\left\{ \Vert Ax-b\Vert^{2}+\lambda \Vert
L(x-x_{0})\Vert^{2}\right\} ,  \label{GenTikh}
\end{equation}%
where $L\in \mathbb{R}^{P\times N}$ is the regularization matrix, $\lambda
>0 $ is the regularization parameter and $x_{0}\in \mathbb{R}^{N}$ is an
initial guess for the solution. We denote the solution of the problem (\ref%
{GenTikh}) by $x_{\lambda }$. When $L=I_{N}$ (the identity matrix of order $%
N $) and $x_{0}=0$, the problem is said to be in standard form. In this
paper the norm $\Vert \cdot \Vert $ is always the Euclidean one. The use of
a regularization matrix different from the identity may improve the quality
of the reconstruction obtained by (\ref{GenTikh}), especially when one wants
to enhance some known features of the solution. In many situations, $L$ is
taken as a scaled finite differences approximation of a derivative operator
(cf. Section \ref{sect:NumExp}).

A proper choice of the regularization parameter is crucial, since it
specifies the amount of regularization to be imposed. Many techniques have
been developed in order to set the regularization parameter in (\ref{GenTikh}%
), we cite \cite{Luk,ParCh} for a review of the classical ones along with
some more recent ones. Here, we are concerned with the discrepancy
principle, that suggests to set the parameter $\lambda $ such that the
nonlinear equation
\begin{equation*}
\Vert b-Ax_{\lambda }\Vert =\eta \Vert e\Vert ,\quad \eta \gtrsim 1,
\end{equation*}%
is satisfied. Of course this strategy can be applied only if a fairly
accurate approximation of the quantity $\Vert e\Vert $ is known.

Denoting by $x_{m,\lambda }$ the approximation of $x_{\lambda }$ computed at
the $m$-th step of a certain iterative method applied to (\ref{GenTikh}),
and by $\phi _{m}(\lambda )=\Vert b-Ax_{m,\lambda }\Vert $ the corresponding
discrepancy, each nonlinear solver for the equation%
\begin{equation}
\phi _{m}(\lambda )=\eta \Vert e\Vert ,  \label{de}
\end{equation}%
leads to a parameter choice rule associated with the iterative process. The
basic idea of this paper, in which we assume $\Vert e\Vert $ to be unknown,
is to consider (if possible) the approximation $\phi _{k}(0)\approx \Vert
e\Vert $, where $k<m$, and then to solve
\begin{equation}
\phi _{m}(\lambda )=\eta \phi _{k}(0),  \label{ad}
\end{equation}%
with respect to $\lambda $. The use of (\ref{ad}) as a parameter choice rule
is motivated by the fact that many iterative solvers for $Ax=b$ produce
approximations $x_{m}=x_{m,0}$ whose corresponding residual $\Vert
b-Ax_{m}\Vert $ tends to stagnate around $\Vert e\Vert $. In other words,
the information about the noise level can be recovered during the iterative
process. Moreover, in many situations, the computational effort of the
algorithm that 
delivers $x_{m,\lambda }$ can be exploited for forming $x_{m,0}$ (or
viceversa). For this reason, we may refer to any iterative process which
simultaneously uses $x_{m}$ to approximate $\Vert e\Vert $ and solves (\ref%
{ad}) to compute $x_{m,\lambda }$ as an embedded approach.

In this paper we are mainly interested in solving (\ref{GenTikh}) by means
of the so-called Arnoldi-Tikhonov methods (originally introduced in \cite%
{ATfirst} for the standard form regularization), which are based on the orthogonal
projection of (\ref{GenTikh}) onto the Krylov subspaces $\mathcal{K}%
_{m}(A,b)=\mathrm{span}\{b,Ab,\dots ,A^{m-1}b\}$ of increasing dimensions.
As well known, these methods typically show a fast superlinear convergence
when applied to discrete ill-posed problems, and hence they are particularly
attractive for large scale problems. Dealing with this kind of methods,
efficient algorithms based on the solution of (\ref{de}) have been
considered in \cite{LR} and \cite{zerof}. More recently, in \cite{auto} a
very simple strategy for solving (\ref{de}), based on the linearization of $%
\phi _{m}(\lambda )$, has been presented. In this paper we extend the latter
approach by considering the approximation $\phi _{m-1}(0)\approx \Vert
e\Vert $ where, in this setting, $\phi _{m-1}(0)$ is just the norm of the
GMRES residual computed at the previous iteration.

The paper is organized as follows. In Section \ref{sec:GAT} we survey the
basic features of the Arnoldi-Tikhonov methods.
In Section 3 we review the linearization technique described in \cite{auto},
and in Section 4 we explain the parameter choice rule based on an embedded
approach, also giving a theoretical justification in the Arnoldi-Tikhonov
case. In the first part of Section 5 we write down the algorithm, in order
to summarize the new method and to better describe some practical details;
the remaining parts are devoted to display the results of some of the
performed numerical tests. In the Appendix, we prove a theorem used
in Section 4.

\section{The Arnoldi-Tikhonov Method}

\label{sec:GAT}

The Arnoldi-Tikhonov (AT) method was first proposed in \cite{ATfirst} with
the basic aims of reducing the problem (\ref{GenTikh}) (in the particular
case $L=I_{N}$ and $x_{0}=0$) to a problem of much smaller dimension and to
avoid the use of $A^{T}$ as in Lanczos type methods (see e.g. \cite{OS}).
Then, in \cite{auto,HostReich,NR} the method has been extended to work with
a general $L\in \mathbb{R}^{P\times N}$ and $x_{0}$. Assuming $x_{0}=0$
(this assumption will hold throughout the paper), we consider the Krylov
subspaces
\begin{equation}
\mathcal{K}_{m}(A,b)=\mathrm{span}\{b,Ab,\dots ,A^{m-1}b\},\;m\geq 1.
\label{KryAb}
\end{equation}%
In order to construct an orthonormal basis for this Krylov subspace we can
use the Arnoldi algorithm \cite{SaadBook}, which leads to the associated
decomposition
\begin{eqnarray}
AW_{m} &=&W_{m}H_{m}+h_{m+1,m}w_{m+1}e_{m}^{T}  \label{ArnoldiFOM} \\
{} &=&W_{m+1}\bar{H}_{m},  \label{Arnoldi}
\end{eqnarray}%
where $W_{m+1}=[w_{1},...,w_{m+1}]\in \mathbb{R}^{N\times (m+1)}$ has
orthonormal columns that span the Krylov subspace $\mathcal{K}_{m+1}(A,b)$,
and $w_{1}=b/\left\Vert b\right\Vert $. The matrices $H_{m}\in \mathbb{R}%
^{m\times m}$ and $\bar{H}_{m}\in \mathbb{R}^{(m+1)\times m}$ are upper
Hessenberg.

The AT method searches for approximations $x_{m,\lambda }$ of the solution
of problem (\ref{GenTikh}) belonging to $\mathcal{K}_{m}(A,b)$. Therefore,
replacing $x=W_{m}y$, $y\in \mathbb{R}^{m}$, into (\ref{GenTikh}), yields
the reduced minimization problem
\begin{equation}
y_{m,\lambda }=\arg \min_{y\in \mathbb{R}^{m}}\left\{ \left\Vert \bar{H}%
_{m}y-c\right\Vert ^{2}+\lambda \left\Vert LW_{m}y\right\Vert ^{2}\right\} ,
\label{TikhPr}
\end{equation}%
where $c=\Vert b\Vert e_{1}$, being $e_{1}$ the first vector of the
canonical basis of $\mathbb{R}^{m+1}$. The above problem is equivalent to
\begin{equation}
y_{m,\lambda }=\arg \min_{y\in \mathbb{R}^{m}}\left\Vert \left(
\begin{array}{c}
\bar{H}_{m} \\
\sqrt{\lambda }LW_{m}%
\end{array}%
\right) y-\left(
\begin{array}{c}
c \\
0%
\end{array}%
\right) \right\Vert ^{2}.  \label{GAT-LS}
\end{equation}%
Obviously, $y_{m,\lambda }$ is also the solution of the normal equation%
\begin{equation}
(\bar{H}_{m}^{T}\bar{H}_{m}+\lambda W_{m}^{T}L^{T}LW_{m})y_{m,\lambda }=\bar{%
H}_{m}^{T}c.  \label{GAT_NE}
\end{equation}%
We remark that, when dealing with standard form problems ($L=I_{N}$ and $%
x_{0}=0$), the Arnoldi-Tikhonov formulation considerably simplifies thanks
again to the orthogonality of the columns of $W_{m}$ and, instead of (\ref%
{GAT-LS}), we can consider
\begin{equation}
y_{m,\lambda }=\arg \min_{y\in \mathbb{R}^{m}}\left\Vert \left(
\begin{array}{c}
\bar{H}_{m} \\
\sqrt{\lambda }I_{m}%
\end{array}%
\right) y-\left(
\begin{array}{c}
c \\
0%
\end{array}%
\right) \right\Vert ^{2}.  \label{GAT-LS-ST}
\end{equation}%
In (\ref{GAT-LS-ST}), the dimension of the problem is fully reduced because
at each iteration we deal with a $(2m+1)\times m$ matrix. On the other side,
considering (\ref{GAT-LS}), there is still track of the original dimensions
of the problem. Anyway, since the AT method can typically recover a meaningful
approximation of the exact solution after just a few iterations of the
Arnoldi algorithm have been performed, the computational cost is still low.
Assuming that $P\leq N$ in (\ref{GenTikh}) and defining a new
matrix $L$ obtained by appending $N-P$ zero rows to the original one, we can
also consider the following new formulation
\begin{equation}
y_{m}=\arg \min_{y\in \mathbb{R}^{m}}\left\Vert \left(
\begin{array}{c}
\bar{H}_{m} \\
\sqrt{\lambda }L_{m}%
\end{array}%
\right) y-\left(
\begin{array}{c}
c \\
0%
\end{array}%
\right) \right\Vert ^{2},\quad \text{where}\quad L_{m}=W_{m}^{T}LW_{m}.
\label{GAT-LS-red}
\end{equation}%
The above problem is not equivalent to (\ref{GAT-LS}) anymore, but can
be justified by the fact that $L_{m}$ is the orthogonal projection of
$L$ onto $\mathcal{K}_{m}(A,b)$, and hence, in some sense, $L_{m}$ inherits
the properties of $L$ (see \cite{new} for a discussion).

\section{The parameter choice strategy}

As said in the Introduction, the discrepancy principle is a well-known and
quite successful parameter selection strategy that, when applied to Tikhonov
regularization method (\ref{GenTikh}), prescribes to choose the
regularization parameter $\lambda >0$ such that $\Vert Ax_{\lambda }-b\Vert
=\eta \Vert e\Vert$, where the parameter $\eta$ is greater than $1$, though
very close to it.

An algorithm exploiting the discrepancy principle has been first considered
for the Arnoldi-Tikhonov method in \cite{LR}, where the authors suggest
to solve, at each iteration $m$, the nonlinear equation
\begin{equation}
\phi _{m}(\lambda ):=\Vert \bar{H}_{m}y_{m,\lambda }-c\Vert =\eta \Vert
e\Vert ,  \label{phi_m}
\end{equation}%
employing a special zero-finder described in \cite{zerof}. In order to
decide when to stop the iterations, a preliminary condition should be
satisfied and then some adjustments should be made.

Considering the normal equations associated to (\ref{GAT-LS-red}), we write
\begin{equation}
\phi _{m}(\lambda )=\Vert c-\bar{H}_{m}(\bar{H}_{m}^{T}\bar{H}_{m}+\lambda
L_{m}^{T}L_{m})^{-1}\bar{H}_{m}^{T}c\Vert .  \label{phi_mNE}
\end{equation}%
Denoting by $r_{m}=b-Ax_{m}$ the GMRES residual, we have that $\phi
_{m}(0)=\left\Vert r_{m}\right\Vert $. In this setting, in \cite{auto} the
authors solve (\ref{phi_m}) after considering the linear approximation
\begin{equation}  \label{LinApprox}
\phi _{m}(\lambda )\approx \phi _{m}(0)+\lambda \beta _{m},
\end{equation}
where, at each iteration, the scalar $\beta _{m}$ is defined by the ratio
\begin{equation}
\beta _{m}=\frac{\phi _{m}(\lambda _{m-1})-\phi _{m}(0)}{\lambda _{m-1}}.
\label{beta}
\end{equation}%
In (\ref{beta}), $\phi _{m}(\lambda _{m-1})$ is obtained by solving the $m$%
-dimensional problem (\ref{GAT-LS-red}) using the parameter $\lambda=\lambda
_{m-1}$, which is computed at the previous step.

Therefore, to select $\lambda=\lambda _{m}$ for the next step of the
Arnoldi-Tikhonov algorithm, we can approximate $\phi_m(\lambda_m)$ by (\ref%
{LinApprox}) and impose
\begin{equation}
\phi _{m}(\lambda _{m})=\eta \Vert e\Vert.  \label{phi}
\end{equation}%
Substituting in the linear approximation of $\phi_m(\lambda_m)$ the
expression derived in (\ref{beta}), and using the condition (\ref{phi}), we
obtain
\begin{equation}
\lambda _{m}=\frac{\eta \Vert e\Vert -\phi _{m}(0)}{\phi _{m}(\lambda
_{m-1})-\phi _{m}(0)}\lambda _{m-1}\,.  \label{lambdaNEW}
\end{equation}%
When $\phi _{m}(0)>\eta \Vert e\Vert $, formula (\ref{lambdaNEW}) produces a
negative value for $\lambda _{m}$. Thus, in order to keep $\lambda _{m}>0$,
we consider the relation
\begin{equation}
\lambda _{m}=\left\vert \frac{\eta \Vert e\Vert -\phi _{m}(0)}{\phi
_{m}(\lambda _{m-1})-\phi _{m}(0)}\right\vert \lambda _{m-1}.
\label{lambdaABS}
\end{equation}%
In this procedure, $\lambda _{0}$ must be set to an initial value by the
user, but the numerical experiments show that this strategy is very robust
with respect to this choice (typically one may set $\lambda _{0}=1$).

\begin{remark}
\label{rem:GMRESres} We remark that the use of the absolute value in (\ref%
{lambdaABS}) can be avoided by forcing initially $\lambda =0$, i.e., working
with the GMRES, and then switching to the AT method equipped with (\ref%
{lambdaNEW}) as soon as $\phi _{m}(0)<\eta \Vert e\Vert $.
\end{remark}

In \cite{auto} this scheme has been called secant-update method , since at each
iteration of the Arnoldi algorithm it basically performs just one step of a
secant-like zero finder applied to the equation $\phi _{m}(\lambda )=\eta
\Vert e\Vert $. Numerically, formula (\ref{lambdaABS}) is very stable, in
the sense that after the discrepancy principle is satisfied, $\lambda _{m}$
is almost constant for growing values of $m$.

\section{Exploiting the GMRES residual}

We now try to generalize the secant-update approach, dropping the hypothesis
that the quantity $\Vert e\Vert $ is available. In this situation, one
typically employs other well-known techniques, such as the L-curve criterion
or the Generalized Cross Validation (GCV); both have already been used in
connection with the Arnoldi-Tikhonov or Lanczos-hybrid methods \cite%
{ATfirst,wGCV,KO,new}. The strategy we are going to describe is to be
considered different since we still want to apply the discrepancy principle,
starting with no information on $\Vert e\Vert $ and trying to recover an
estimate of it during the iterative process.

Our basic assumption is that, after just a few iterations of the Arnoldi
algorithm, the norm of the residual associated to the GMRES method lies
around the threshold $\Vert e\Vert $ and, despite being slightly decreasing,
stabilizes during the following iterations (cf. Figure \ref{f04}). This
motivates the use of the following strategy to choose the regularization
parameter at the $m$-th iteration
\begin{equation}
\lambda _{m}=\frac{\eta \phi _{m-1}(0)-\phi _{m}(0)}{\phi _{m}(\lambda
_{m-1})-\phi _{m}(0)}\lambda _{m-1},\quad \eta >1,  \label{NEWupdate}
\end{equation}%
where we have replaced the quantity $\Vert e\Vert $ in (\ref{lambdaABS}) by $%
\phi _{m-1}(0)=\Vert r_{m-1}\Vert $. We remark that, from a theoretical
point of view, the formula (\ref{NEWupdate}) cannot produce negative values
since $\phi_m(0)=\Vert r_{m}\Vert \leq \Vert r_{m-1}\Vert=\phi_{m-1}(0) $
and $\phi _{m}(\lambda )$ is an increasing function with respect to $\lambda
$. In what follows we provide a theoretical justification for this approach,
giving also some numerical experiments using test problems taken from \cite%
{H1}; in the first subsection we focus on the case $b=b^{ex}$, while in the
second subsection we treat the case $b=b^{ex}+e$.

\subsection{The unperturbed problem}

Thanks to a number of results in literature (see e.g. \cite{Moret}), we know
that the GMRES exhibits superlinear convergence when solving problems in
which the singular values rapidly decay to 0. Indeed, in this situation, the
Krylov subspaces tend to become $A$-invariant after few iterations. In general,
the fast convergence of a Krylov subspace method applied to an ill-posed
system (\ref{sys}) can be explained by monitoring the behavior of the
sequence $\left\{ h_{m+1,m}\right\} _{m}$. In fact, it is well known that the
GMRES residual is related with the FOM residual $\rho _{m}$ as follows \cite[%
Chapter 6]{SaadBook}%
\begin{equation}
\left\Vert r_{m}\right\Vert \leq h_{m+1,m}\left\vert
e_{m}^{T}H_{m}^{-1}c\right\vert =\left\Vert \rho _{m}\right\Vert ,
\label{GMRES-FOM}
\end{equation}%
where $H_{m}$ is as in (\ref{ArnoldiFOM}) and $c=\Vert b^{ex}\Vert e_{1}\in
\mathbb{R}^{m}$. Thanks to the relation (see \cite{Me})%
\begin{equation*}
\left\Vert r_{m}\right\Vert ^{2}=\frac{1}{\frac{1}{\left\Vert \rho
_{m}\right\Vert ^{2}}+\frac{1}{\left\Vert r_{m-1}\right\Vert ^{2}}},
\end{equation*}%
which expresses the well known peak-plateau phenomenon, we can conclude that
when the FOM solutions do not explode the GMRES residuals decay as the
quantities $h_{m+1,m}$. The following theorem (proved in the Appendix) gives
us an estimate for the quantities $\left\{ h_{m+1,m}\right\} _{m}$ whenever
we work with the exact right hand side $b^{ex}$, and $A$ is assumed to be
severely ill-conditioned, that is, with singular values which decay
exponentially (cf. \cite{Hofm}). In Figure \ref{f01} we report a couple of
numerical experiments.

\begin{theorem}
\label{th1} Assume that $A$ has full rank with singular values of the type $%
\sigma _{j}=O(e^{-\alpha j})$ ($\alpha >0$) and that $b^{ex}$ satisfies the
Discrete Picard Condition, that is, $\left\vert u_{j}^{T}b^{ex}\right\vert
\sim \sigma _{j}$, where $u_{j}$ is the $j$-column of the matrix $U$ of (\ref%
{SVD}). Then if $b^{ex}$ is the starting vector of the Arnoldi process we
have
\begin{equation}
h_{m+1,m}=O\left( m^{3/2}\sigma _{m}\right) .  \label{rr3}
\end{equation}
\end{theorem}

\begin{figure}[tbp]
\centering
\begin{tabular}{cc}
\includegraphics[width=0.45\textwidth]{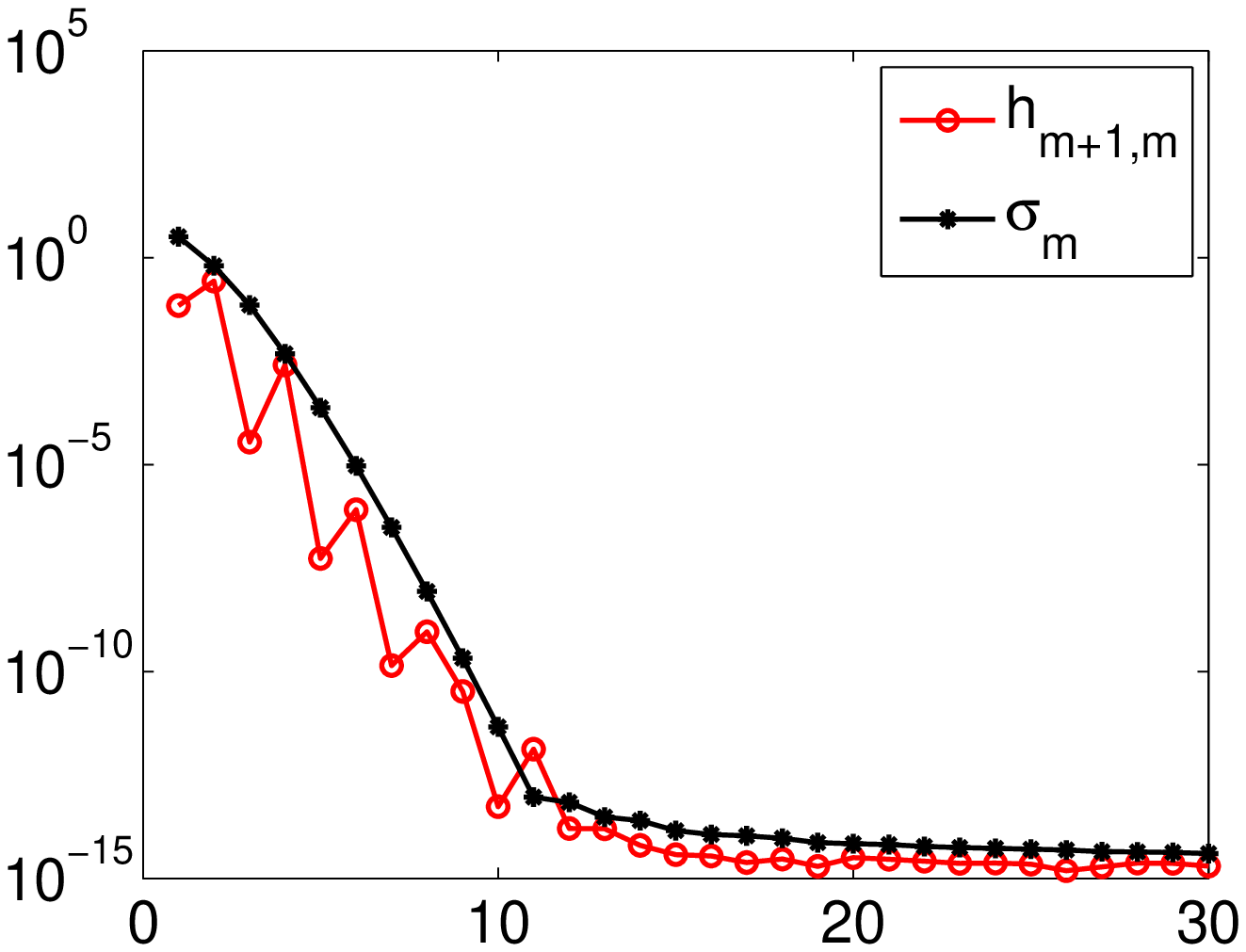} & %
\includegraphics[width=0.45\textwidth]{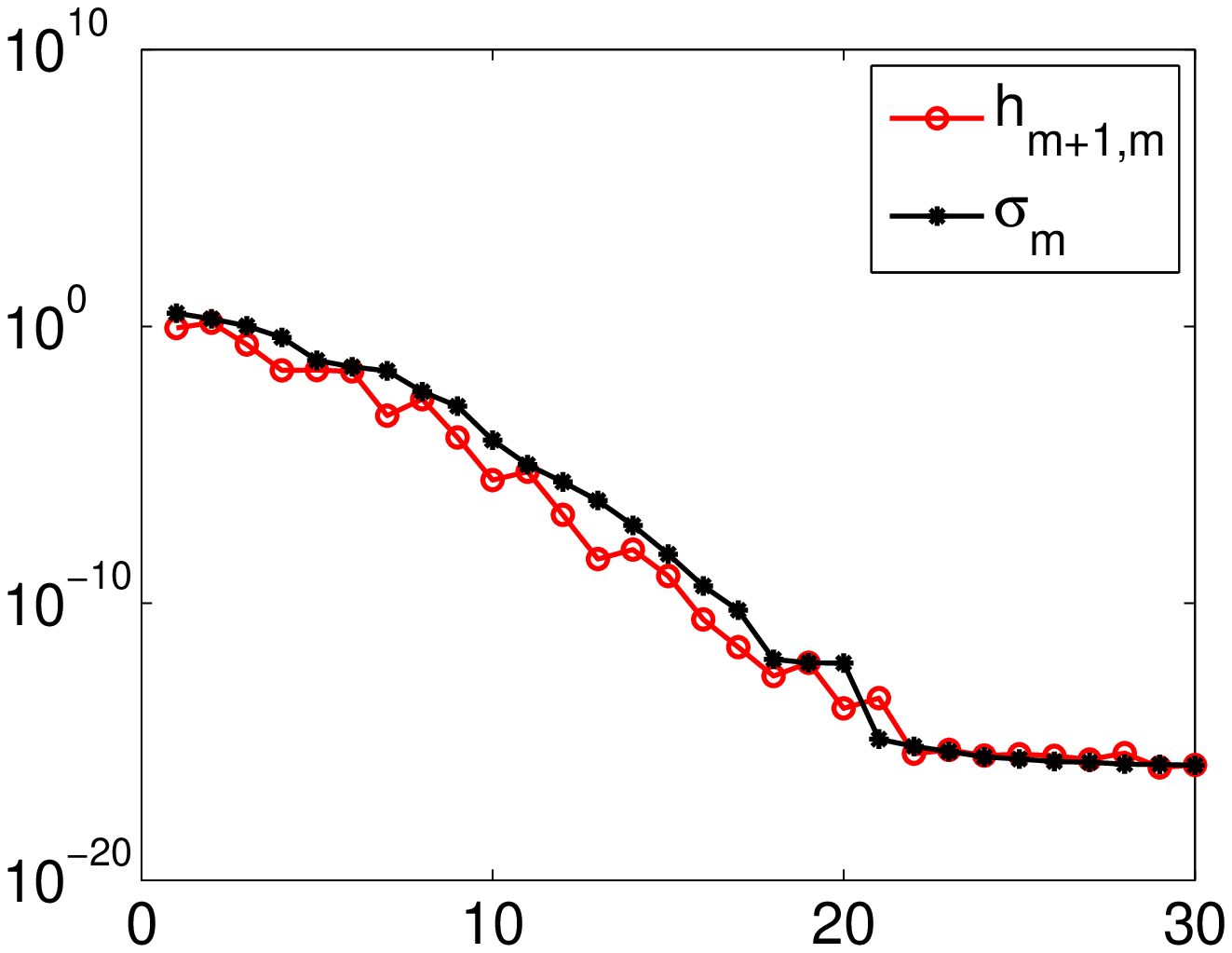}%
\end{tabular}%
\caption{Behavior of the sequences $\left\{ h_{m+1,m}\right\} _{m}$ and $%
\left\{ \protect\sigma _{m}\right\} _{m}$ for the test problems \texttt{baart%
} (left) and \texttt{shaw } (right) from \protect\cite{H1}.}
\label{f01}
\end{figure}

The following result follows immediately from Theorem \ref{th1} and (\ref%
{GMRES-FOM}).

\begin{corollary}
\label{cor1}Under the hypothesis of Theorem \ref{th1}, assume that there
exist $M$ such that for $m\leq N$%
\begin{equation}
\left\vert e_{m}^{T}H_{m}^{-1}c\right\vert \leq M,  \label{bm2}
\end{equation}%
where $c=\Vert b\Vert e_{1}\in \mathbb{R}^{m}$. Then the GMRES residuals are
of the type%
\begin{equation}
\left\Vert r_{m}^{ex}\right\Vert =O\left( m^{3/2}\sigma _{m}\right) .
\label{bg}
\end{equation}
\end{corollary}

Employing the SVD of the matrix $H_{m}$, that is, $H_{m}=U_{m}^{(m)}\Sigma
_{m}^{(m)}\left( V_{m}^{(m)}\right) ^{T}$, $\Sigma _{m}^{(m)}=\mathrm{diag}%
(\sigma _{1}^{(m)},...,\sigma _{m}^{(m)})$, we have%
\begin{equation*}
H_{m}^{-1}c=V_{m}^{(m)}(\Sigma _{m}^{(m)})^{-1}U_{m}^{(m)T}c,
\end{equation*}%
so that (\ref{bm2}) is satisfied as soon as the Discrete Picard Condition is
inherited in some way by the projected problem. It is known that if $%
\widetilde{\sigma }_{j}^{(m)}$, $j=1,...,m$, are the singular values
approximations arising from the SVD of $\bar{H}_{m}$, then $\widetilde{%
\sigma }_{m}^{(m)}\geq \widetilde{\sigma }_{m+1}^{(m+1)}\geq \sigma _{N}>0$
(cf. \cite{ATfirst}). Since $h_{m+1,m}$ goes rapidly to 0, we also have that
after a few iterations $\sigma _{j}^{(m)}\approx \widetilde{\sigma }%
_{j}^{(m)}$ so that we can expect that $\sigma _{m}^{(m)}\geq \sigma _{N}$.
In general, however, we do not have guarantees that $M$ is small, so that (%
\ref{bg}) may be quantitatively not much useful. Everything is closely
related to the SVD approximation that we can achieve with the Arnoldi
algorithm (see \cite{new} for some theoretical results). It is known that if
the matrix $A$ is highly nonsymmetric, then the SVD approximation may be
poor so that the Discrete Picard Condition may be badly inherited by the
projected problem. Anyway, in Figure \ref{f02} we report the FOM residual
history for some test problems which confirm the behavior described by (\ref%
{bg}).

\begin{figure}[]
\centering
\includegraphics[width=0.50\textwidth]{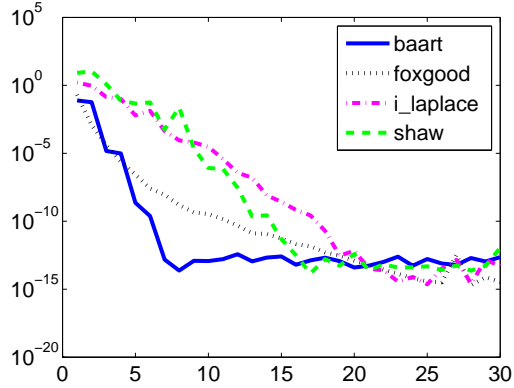}
\caption{FOM residual history for some common test problems taken from
\protect\cite{H1}.}
\label{f02}
\end{figure}

\subsection{The perturbed problem}

When the right-hand side of (\ref{sys}) is affected by noise, we can give
the following preliminary estimate for the norm of the GMRES residual.

\begin{proposition}
\label{prop1}Let $b=b^{ex}+e$ and let $r_{m}^{ex}=p_{m}^{ex}(A)b^{ex}$ be
the residual of the GMRES applied to the system $Ax=b^{ex}$. Assume that for
$m\geq m^{\ast }$, $\left\Vert p_{m}^{ex}(A)\right\Vert \leq \eta ^{\ast }$.
Then the $m$-th residual of the GMRES applied to $Ax=b$ satisfies%
\begin{equation*}
\left\Vert r_{m}\right\Vert \leq \eta \Vert e\Vert ,
\end{equation*}%
where%
\begin{equation*}
\eta =\frac{\left\Vert r_{m^{\ast }}^{ex}\right\Vert }{\Vert e\Vert }+\eta
^{\ast }.
\end{equation*}
\end{proposition}

\begin{proof}
Since $b=b^{ex}+e$ and, thanks to the optimality property of the GMRES
residual,
\begin{equation*}
\left\Vert r_{m}\right\Vert =\min_{p_{m}(0)=1}\left\Vert
p_{m}(A)b\right\Vert \leq \left\Vert p_{m}^{ex}(A)b\right\Vert ,
\end{equation*}%
and hence%
\begin{equation*}
\left\Vert r_{m}\right\Vert\leq\left\Vert
p_m^{ex}(A)b^{ex}\right\Vert+\left\Vert p_m^{ex}(A)e\right\Vert \leq
\left\Vert r_{m}^{ex}\right\Vert +\eta ^{\ast }\Vert e\Vert .
\end{equation*}%
The result follows from $\left\Vert r_{m}^{ex}\right\Vert \leq \left\Vert
r_{m^{\ast }}^{ex}\right\Vert $, which holds for $m\geq m^{\ast}$.
\end{proof}

In the remaining part of this section, we try to give some additional
information about the value of the constant $\eta $ of Proposition \ref%
{prop1}. Let%
\begin{equation*}
\widetilde{V}_{m}=\left[ \frac{b}{\left\Vert b\right\Vert },\frac{Ab}{%
\left\Vert Ab\right\Vert },...,\frac{A^{m-1}b}{\left\Vert
A^{m-1}b\right\Vert }\right] ,\quad \widetilde{V}_{m}^{ex}=\left[ \frac{%
b^{ex}}{\left\Vert b^{ex}\right\Vert },\frac{Ab^{ex}}{\left\Vert
Ab^{ex}\right\Vert },...,\frac{A^{m-1}b^{ex}}{\left\Vert
A^{m-1}b^{ex}\right\Vert }\right] .
\end{equation*}%
With this notations we can write%
\begin{equation*}
\left\Vert r_{m}\right\Vert =\min_{s\in \mathbb{R}^{m+1},s_{1}=0}\left\Vert
b-\widetilde{V}_{m+1}s\right\Vert,
\end{equation*}
where $s_1$ is the first component of vector $s$.

\begin{proposition}
\label{prop2}For the GMRES residual we have%
\begin{equation*}
\left\Vert r_{m}\right\Vert \leq \eta (m)\Vert e\Vert ,
\end{equation*}%
where%
\begin{equation*}
\eta (m)=1+\frac{\left\Vert r_{m}^{ex}\right\Vert +\left\Vert \left(
\widetilde{V}_{m+1}-\widetilde{V}_{m+1}^{ex}\right) s^{ex}\right\Vert }{%
\left\Vert e\right\Vert },
\end{equation*}%
in which $s^{ex}$ ($s_{1}^{ex}=0$) is such that $\left\Vert
r_{m}^{ex}\right\Vert =\left\Vert b-\widetilde{V}_{m+1}^{ex}s^{ex}\right%
\Vert $.
\end{proposition}

\begin{proof}
We have%
\begin{eqnarray*}
\left\Vert r_{m}\right\Vert &=&\min_{s\in \mathbb{R}^{m}}\left\Vert b-%
\widetilde{V}_{m+1}s\right\Vert \leq \left\Vert b-\widetilde{V}%
_{m+1}s^{ex}\right\Vert \\
&=&\left\Vert b^{ex}+e-\widetilde{V}_{m+1}s^{ex}+\widetilde{V}%
_{m+1}^{ex}s^{ex}-\widetilde{V}_{m+1}^{ex}s^{ex}\right\Vert \\
&\leq &\left\Vert r_{m}^{ex}\right\Vert +\left\Vert e\right\Vert +\left\Vert
\left( \widetilde{V}_{m+1}-\widetilde{V}_{m+1}^{ex}\right) s^{ex}\right\Vert
.
\end{eqnarray*}
\end{proof}

The fast decay of the singular values of $A$ ensures that, for $k\geq 1$
(note that $s_{1}^{ex}=0$)
\begin{equation}
\frac{1}{\left\Vert e\right\Vert }\left\Vert \frac{A^{k}b}{\left\Vert
A^{k}b\right\Vert }-\frac{A^{k}b^{ex}}{\left\Vert A^{k}b^{ex}\right\Vert }%
\right\Vert \ll 1,  \label{rrr}
\end{equation}%
so that, whenever $\left\Vert r_{m}^{ex}\right\Vert \approx 0$, we have $%
\eta (m)\approx 1$. Condition (\ref{rrr}) is also at the basis of the
so-called range-restricted approach for Krylov type methods (see \cite{LR}).
We also remark that the relation (\ref{rrr}) can be interpreted as the
discrete analogous of the Riemann-Lebesgue Lemma (see e.g. \cite[p.6]{PCH}),
whenever we assume that the noise $e$ does not involve low frequencies. We
give some examples of this behavior in Figure \ref{f03}.

\begin{figure}[tbp]
\centering
\includegraphics[width=0.50\textwidth]{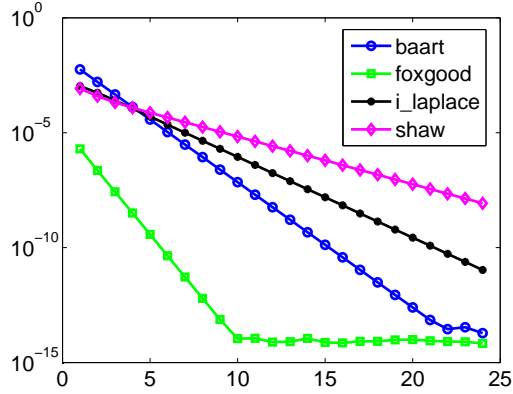}
\caption{Decay of the quantities $\frac{1}{\Vert e\Vert }\left\Vert \frac{%
A^{k}b}{\left\Vert A^{k}b\right\Vert }-\frac{A^{k}b^{ex}}{\left\Vert
A^{k}b^{ex}\right\Vert }\right\Vert $ versus the value of $k\geq 1$. The right-hand side is
affected by 1\% Gaussian noise.}
\label{f03}
\end{figure}

Finally, in Figure \ref{f04} we prove experimentally our main assumption,
that is, $\left\Vert r_{m}\right\Vert \approx \left\Vert e\right\Vert $ for $%
m$ sufficiently large, which justifies the use of formula (\ref{NEWupdate}).

\begin{figure}[]
\centering
\begin{tabular}{cc}
\includegraphics[width=0.40\textwidth]{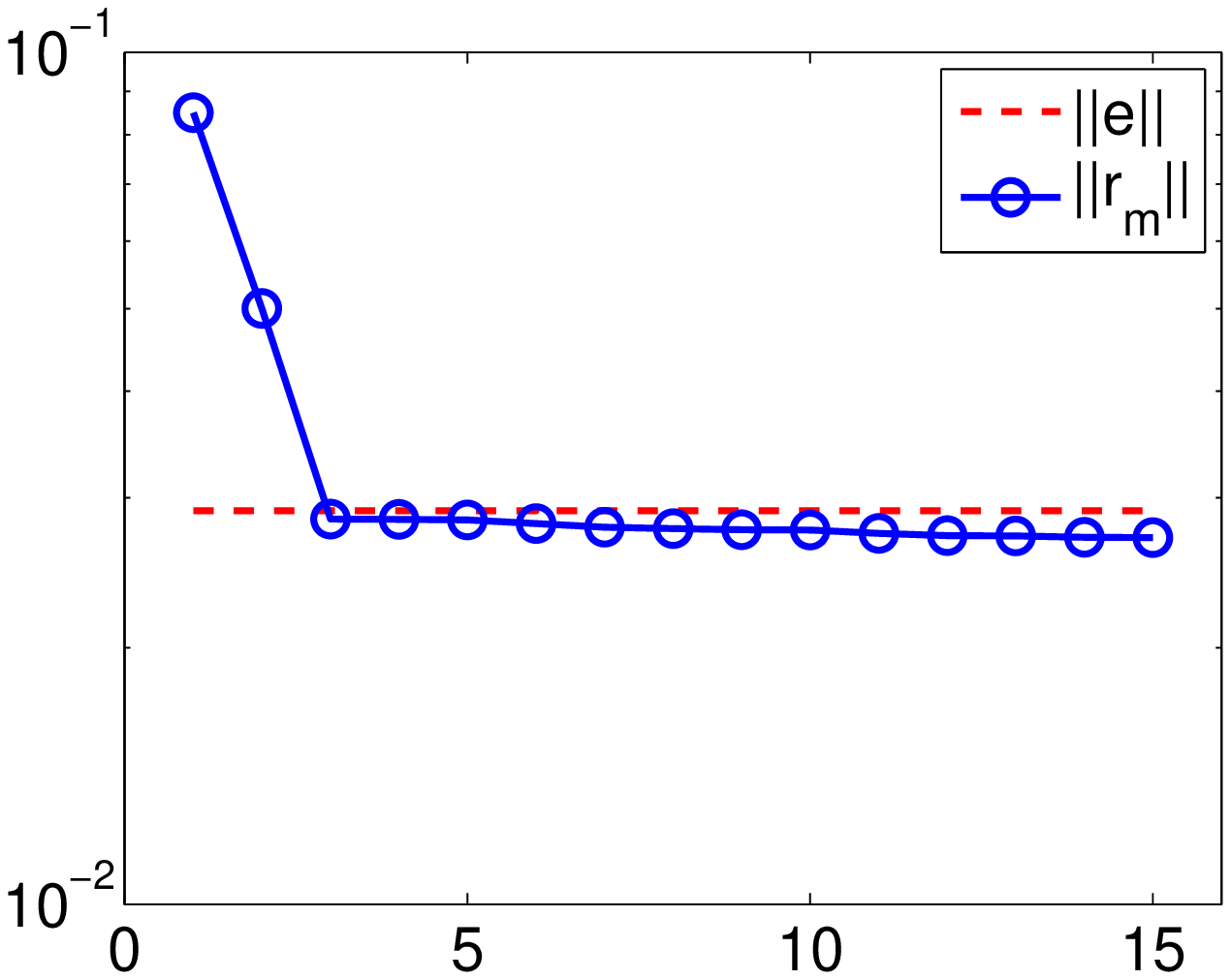} & %
\includegraphics[width=0.40\textwidth]{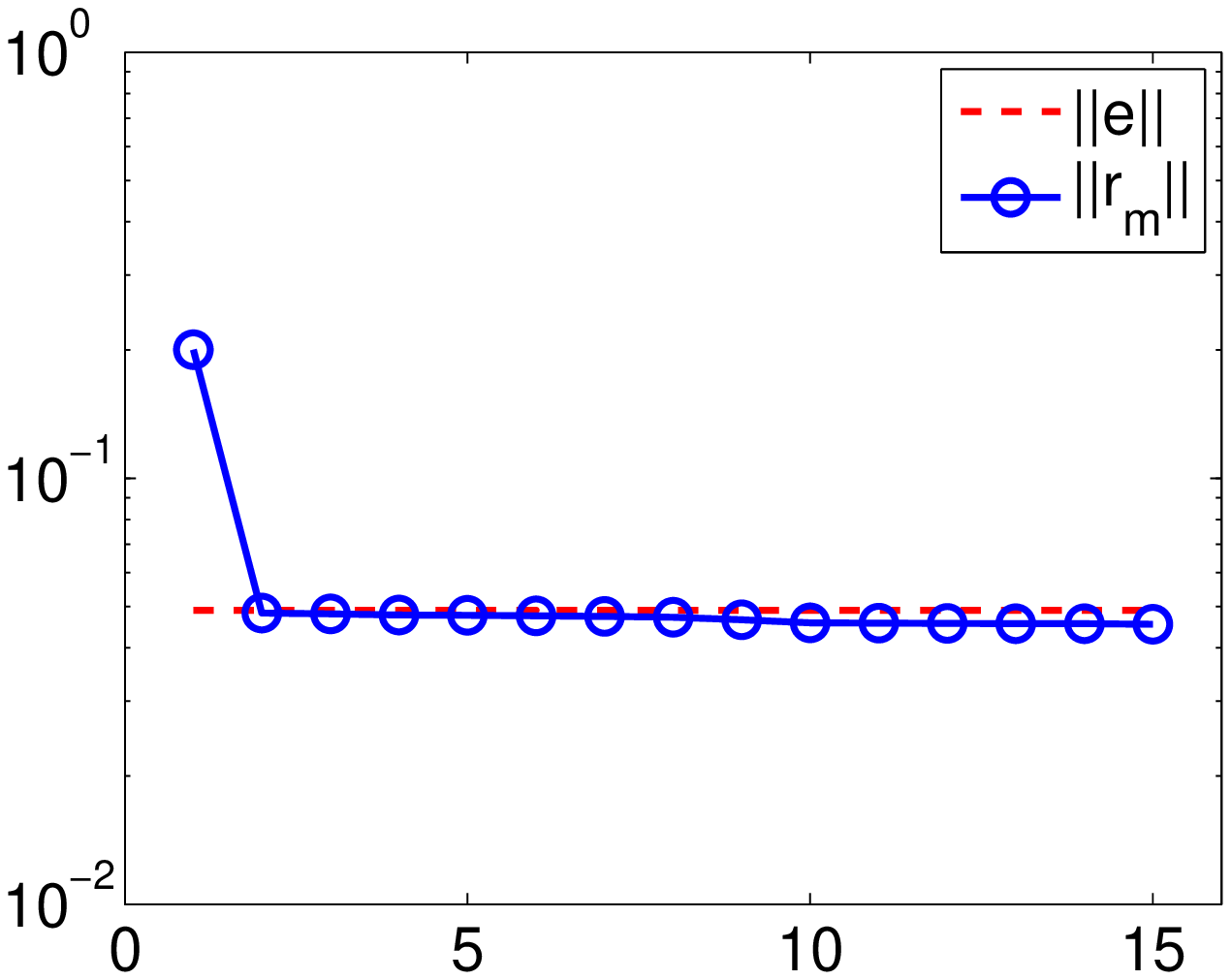} \\
\includegraphics[width=0.40\textwidth]{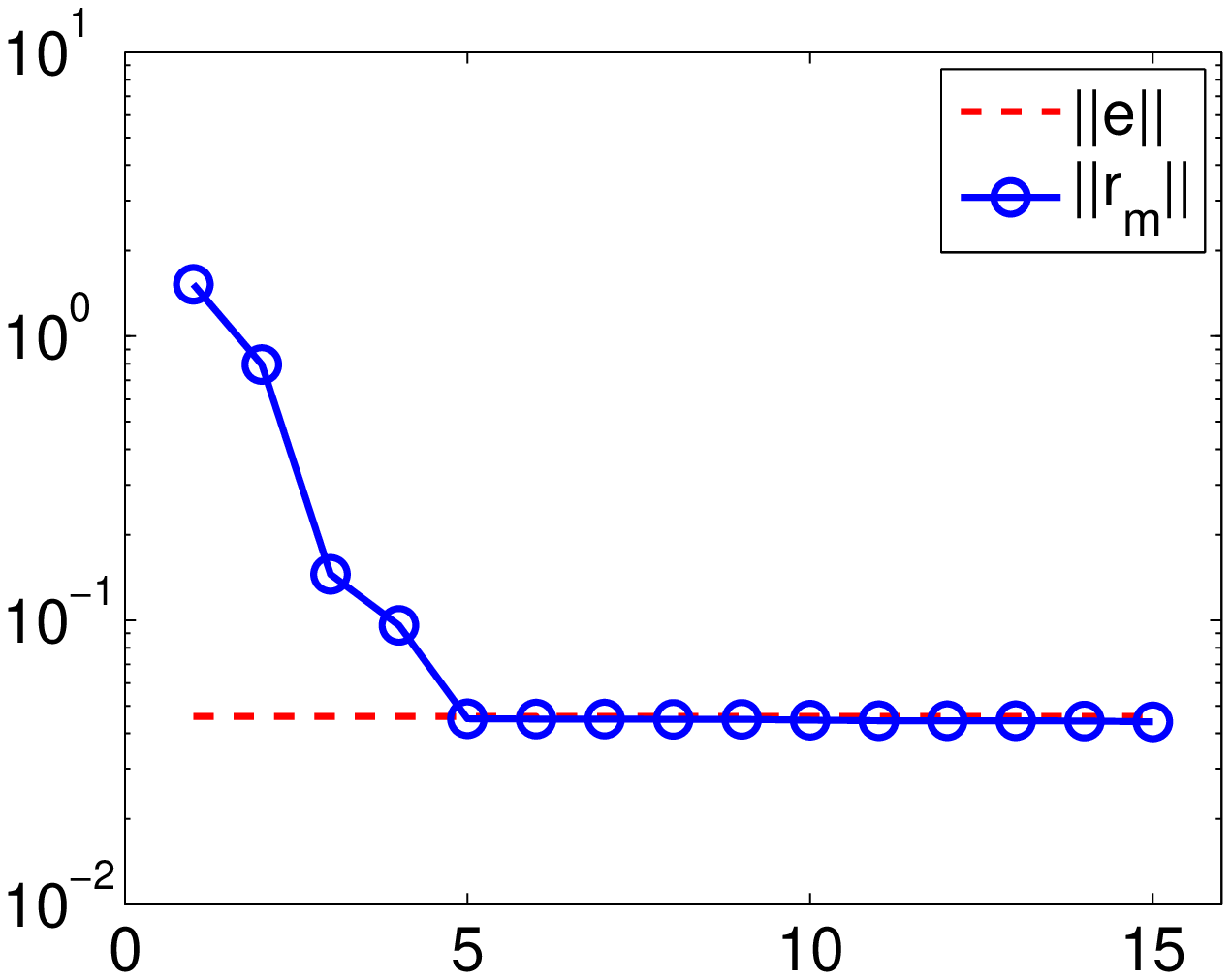} & %
\includegraphics[width=0.40\textwidth]{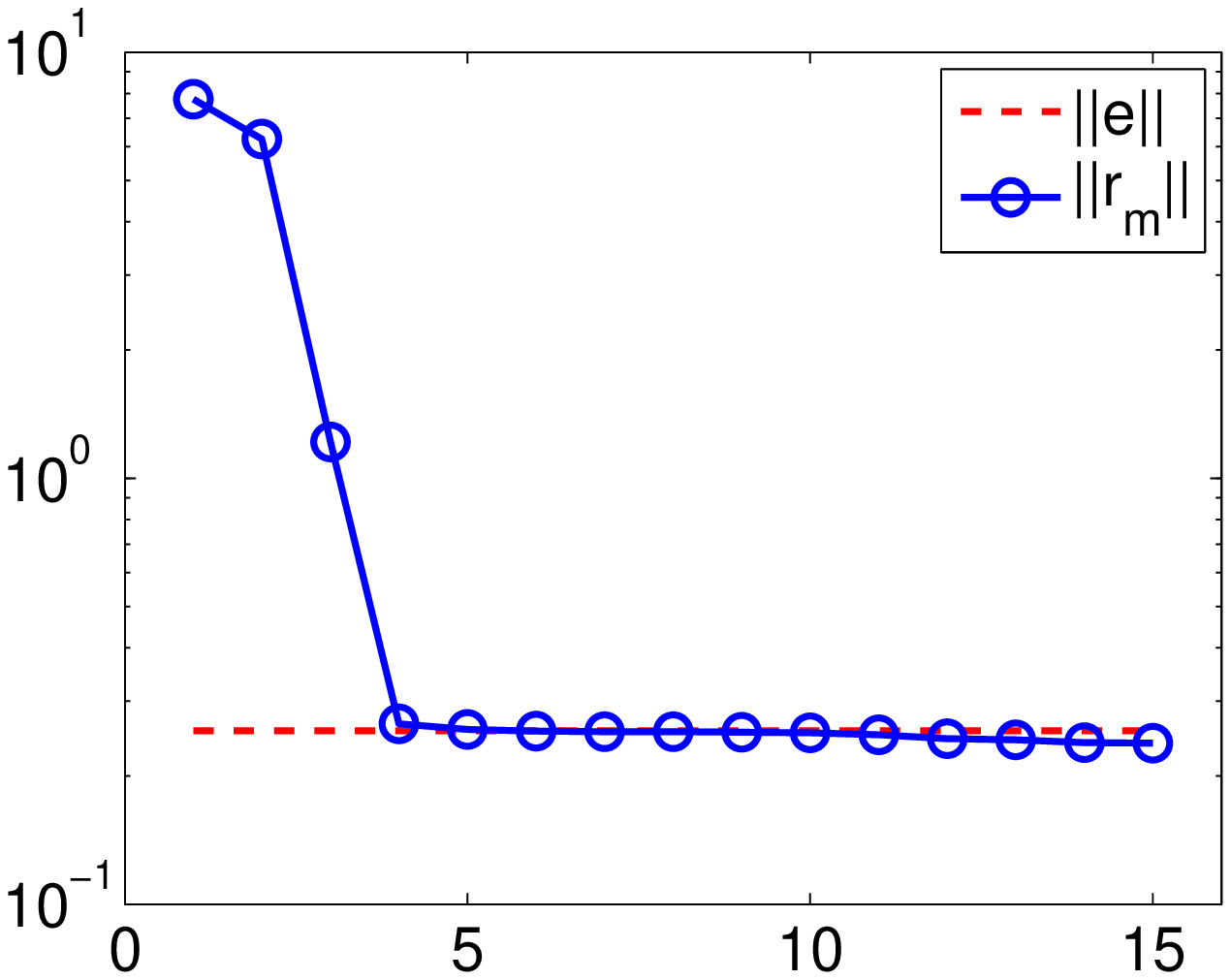}%
\end{tabular}%
.
\caption{GMRES residual history when the right-hand side is affected by 1\%
noise. In clockwise order the problem considered are \texttt{%
baart}, \texttt{foxgood}, \texttt{shaw} and \texttt{i\_laplace}}
\label{f04}
\end{figure}

\section{Algorithm and Numerical Experiments}

\label{sect:NumExp}

Comparing the parameter selection strategies (\ref{lambdaABS}) and (\ref%
{NEWupdate}), we can state that (\ref{NEWupdate}) generalizes the approach
described in Section 3, since no knowledge of $\Vert e\Vert $ is assumed.
However, on the downside, scheme (\ref{lambdaABS}) can simultaneously
determine the value of the regularization parameter at each iteration and
the number of iterations to be performed, while this is no more possible
considering the rule (\ref{NEWupdate}). In order to determine when to stop
the iterations of the Arnoldi algorithm, we have to consider a separate
stopping criterion. Since both $\phi _{m}(\lambda _{m-1})$ and $\Vert
r_m\Vert $ exhibit a stable behavior going on with the iterations, a way to
set $m$ is to monitor when such stability occurs, i.e., to evaluate the
relative difference between the norm of the residuals and the relative
difference between the discrepancy functions. Therefore, once two thresholds
$\tau_{\mathrm{res}}$ and $\tau_{\mathrm{discr}}$ have been set, we decide
to stop the iterations as soon as%
\begin{equation}
\frac{\Vert r_m\Vert -\Vert r_{m-1}\Vert }{\Vert r_{m-1}\Vert}<\tau_{\mathrm{%
res}} ,  \label{stopRES}
\end{equation}%
and
\begin{equation}
\frac{\phi _{m}(\lambda _{m-1})-\phi _{m-1}(\lambda _{m-2})}{\phi
_{m-1}(\lambda _{m-2})}<\tau_{\mathrm{discr}}.  \label{stopDISCR}
\end{equation}%
This approach is very similar to the one adopted in \cite{wGCV} for the GCV
method in a hybrid setting. Also in \cite{ATfirst} the authors decide to
terminate the Arnoldi process when the corners of two consecutive projected
L-curves are pretty close. We can also expect the value of $\lambda _{m}$
obtained at the end of the iterations to be suitable for the original
problem (\ref{GenTikh}).

The method so far described, can be summarized in the following

\begin{algorithm}
\label{Alg1}
\caption{AT method equipped with the parameter choice rule
(\ref{NEWupdate})}
\begin{algorithmic}
\State \textbf{Inputs}: $A$, $b$, $L$, $x_0$, $\lambda_0$, $\eta$, $\tres$, $\tdiscr$
\State \textbf{For} $m=1,2,\dots,$ until (\ref{stopRES}) and (\ref{stopDISCR}) are both fulfilled
\begin{enumerate}
\item
Update $W_{m}$ and $\bar{H}_m$ by the Arnoldi algorithm (\ref{Arnoldi}).
\item
Compute the reduced-dimension GMRES solution $y_{m,0}$ 
(cf. (\ref{TikhPr}) and the corresponding residual $r_m$.
\item
Compute the solution $y_{m,\lambda}$ of (\ref{GAT-LS-red}), taking
\[
\begin{cases}
\lambda=\lambda_{0} & \mathrm{if}\: m=1,2,\\
\lambda=\lambda_{m-1} & \mathrm{otherwise.}
\end{cases}.
\]
\item Compute 
the 
discrepancy $\phi_m(\lambda_{m-1})=\|\bar{H}_my_{m,\lambda_{m-1}}-c\|$.
\item if $m\geq 2$ update $\lambda_m$ by formula (\ref{NEWupdate}).
\end{enumerate}
\State \textbf{end}
\State Compute $x_{m,\lambda_{m-1}}=W_my_{m,\lambda_{m-1}}$.
\end{algorithmic}
\end{algorithm}

To illustrate the behavior of this algorithm, we treat three different kinds
of test problems. All the experiments have been carried out using Matlab
7.10 with $16$ significant digits on a single processor computer (Intel Core
i7). The algorithm is implemented with $\lambda _{0}=1$, $\eta
=1.02$, and $\tau _{\mathrm{res}}=\tau _{\mathrm{discr}}=5\cdot 10^{-2}$.

\subsection{Test problems from \texttt{Regularization Tools}}

We consider again some classical test problems taken from Hansen's \texttt{Regularization Tools}
\cite{H1}. In particular in Figure \ref{F1}, we report the results for
the problems \texttt{baart}, \texttt{shaw}, \texttt{foxgood}, \texttt{%
i\_laplace}; the right-hand side $b$ is affected by additive 0.1\% Gaussian
noise $e$, such that the noise level $\varepsilon =\Vert e\Vert /\Vert
b^{ex}\Vert$ is equal to $10^{-3}$. The dimension of each problem is $N=120$%
. The regularization operator used is the discrete first derivative $L_1$
for \texttt{shaw} and \texttt{i\_laplace}, and the discrete second
derivative $L_2$ for \texttt{baart} and \texttt{foxgood}, augmented with one
or two zero rows respectively, in order to make it square, that is,

\begin{equation}  \label{RegMatr}
L_{1}:=\left(
\begin{array}{cccc}
1 & -1 &  &  \\
& \ddots & \ddots &  \\
&  & 1 & -1 \\
0 & ... & ... & 0%
\end{array}%
\right) ,\quad L_{2}:=\left(
\begin{array}{ccccc}
1 & -2 & 1 &  &  \\
& \ddots & \ddots & \ddots &  \\
&  & 1 & -2 & 1 \\
0 & ... & ... & ... & 0 \\
0 & ... & ... & ... & 0%
\end{array}%
\right) .
\end{equation}

\begin{figure}[tbp]
\centering
\includegraphics[width=0.32\textwidth]{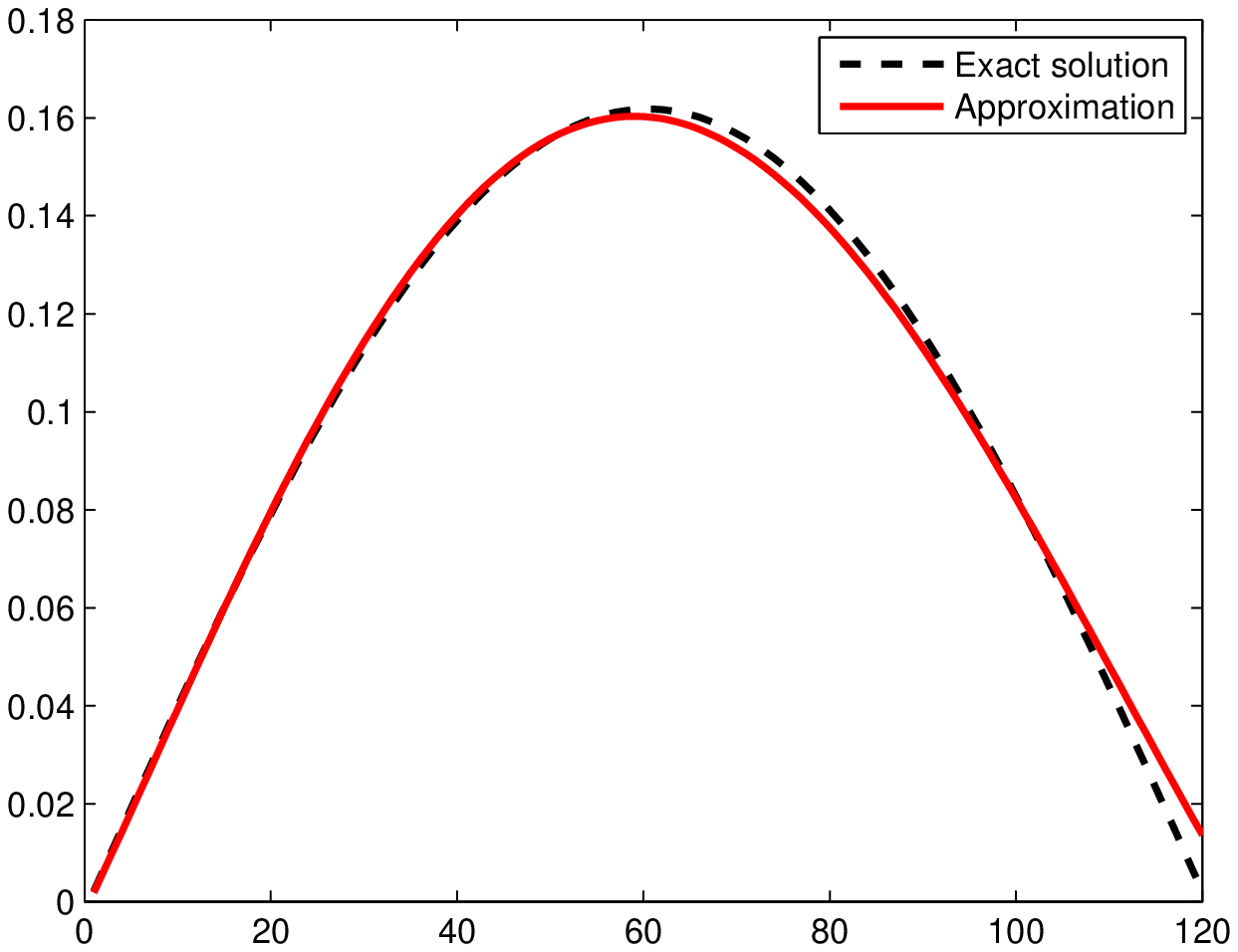} %
\includegraphics[width=0.32\textwidth]{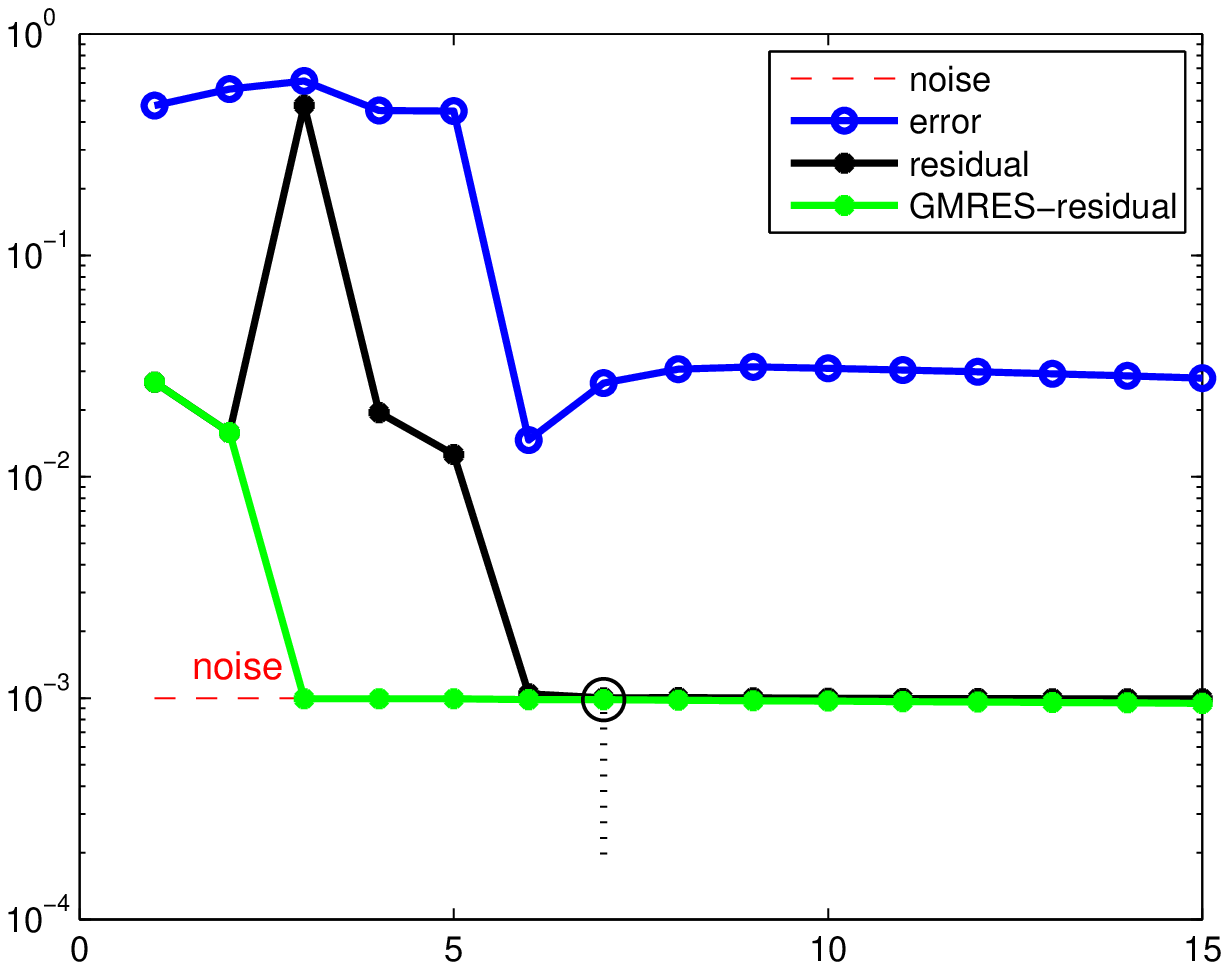} %
\includegraphics[width=0.32\textwidth]{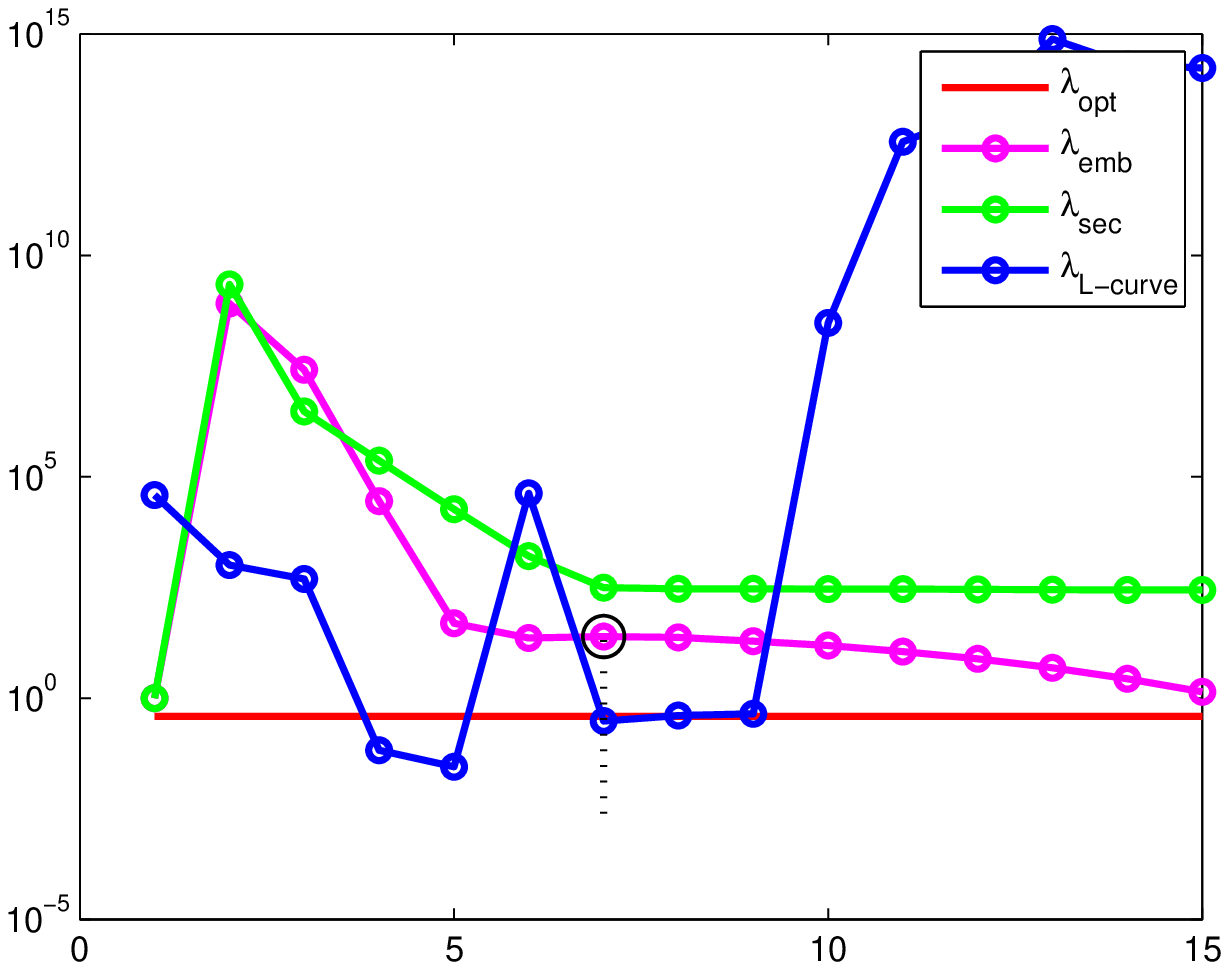} \newline
\par
\includegraphics[width=0.32\textwidth]{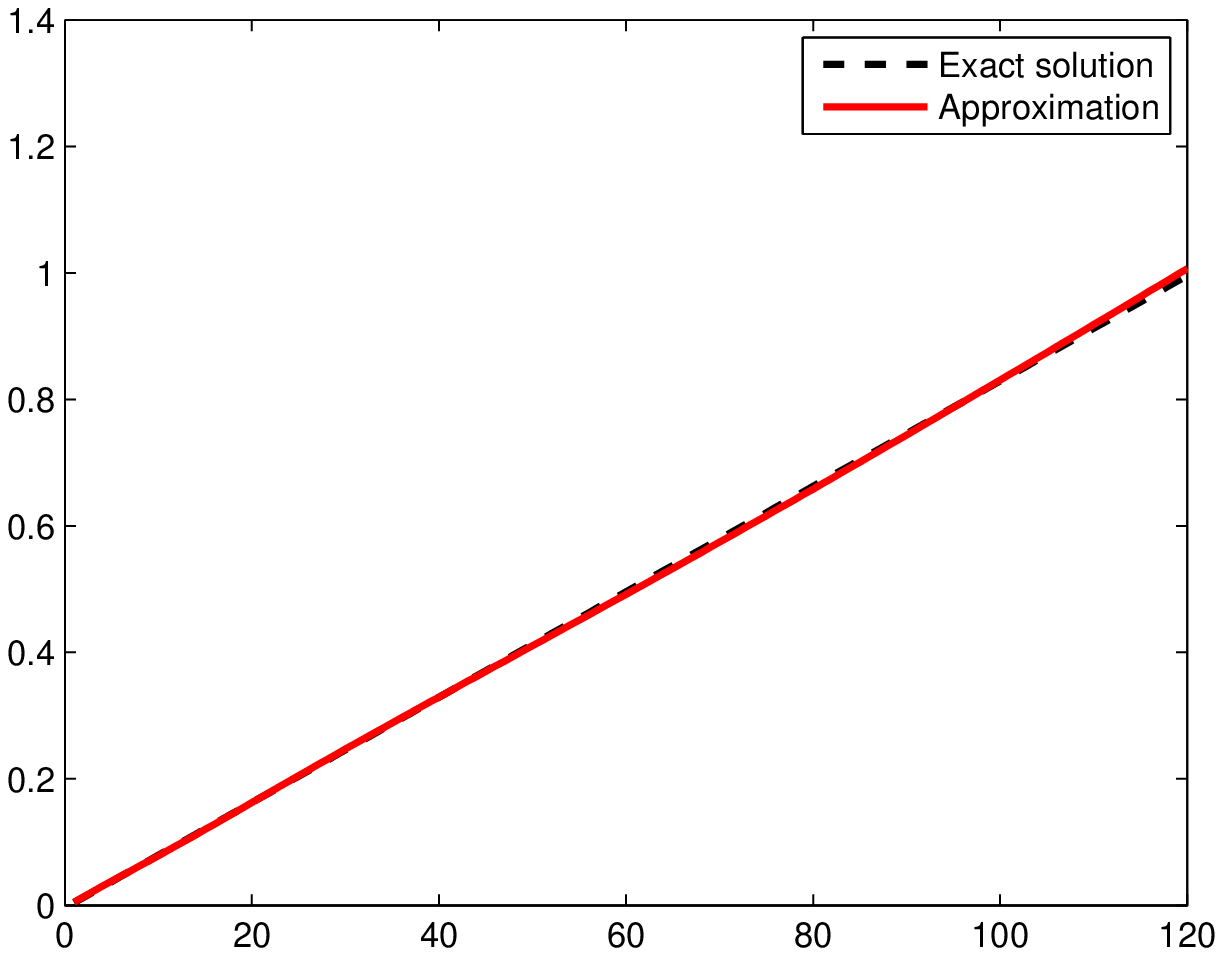} %
\includegraphics[width=0.32\textwidth]{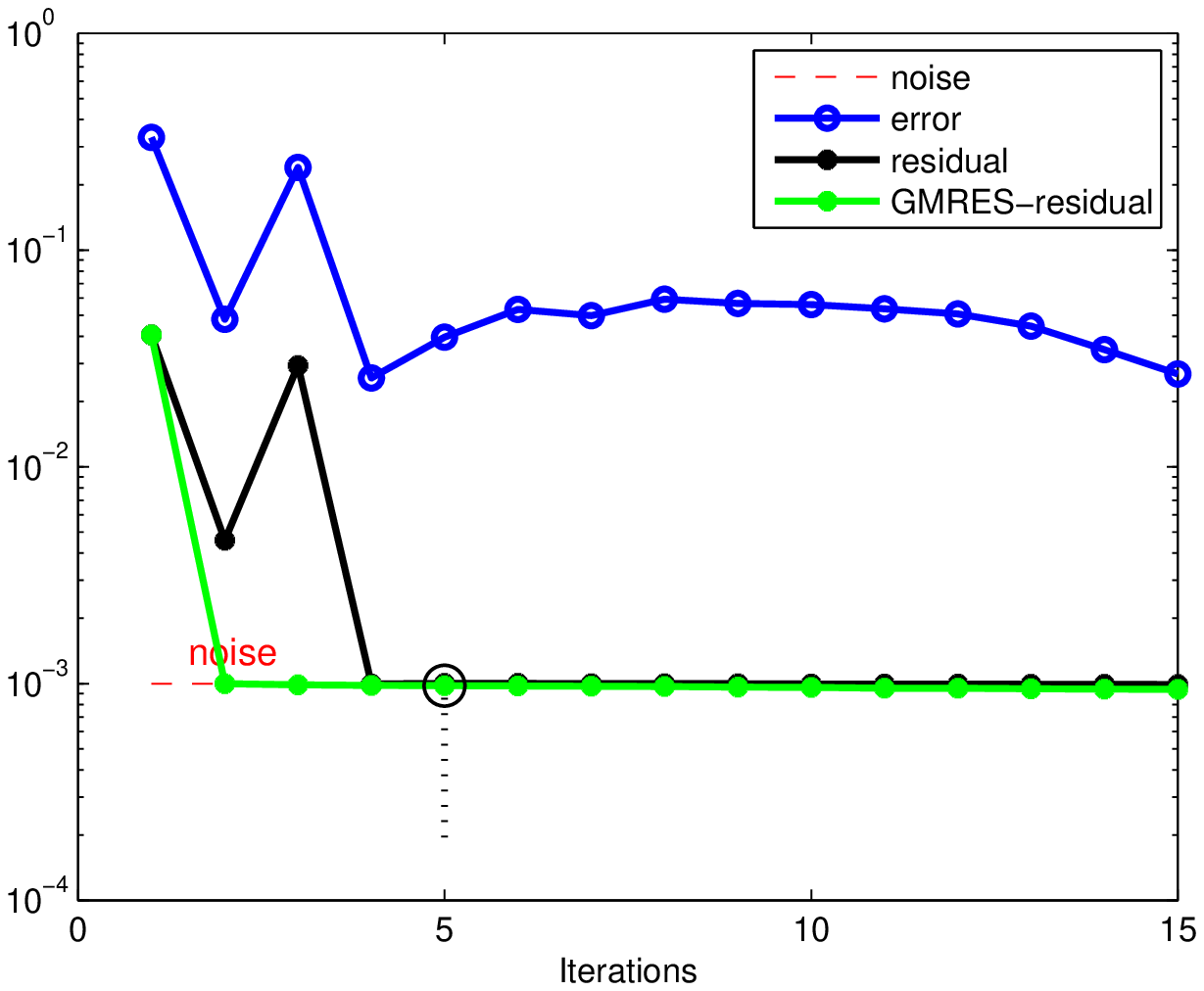} %
\includegraphics[width=0.32\textwidth]{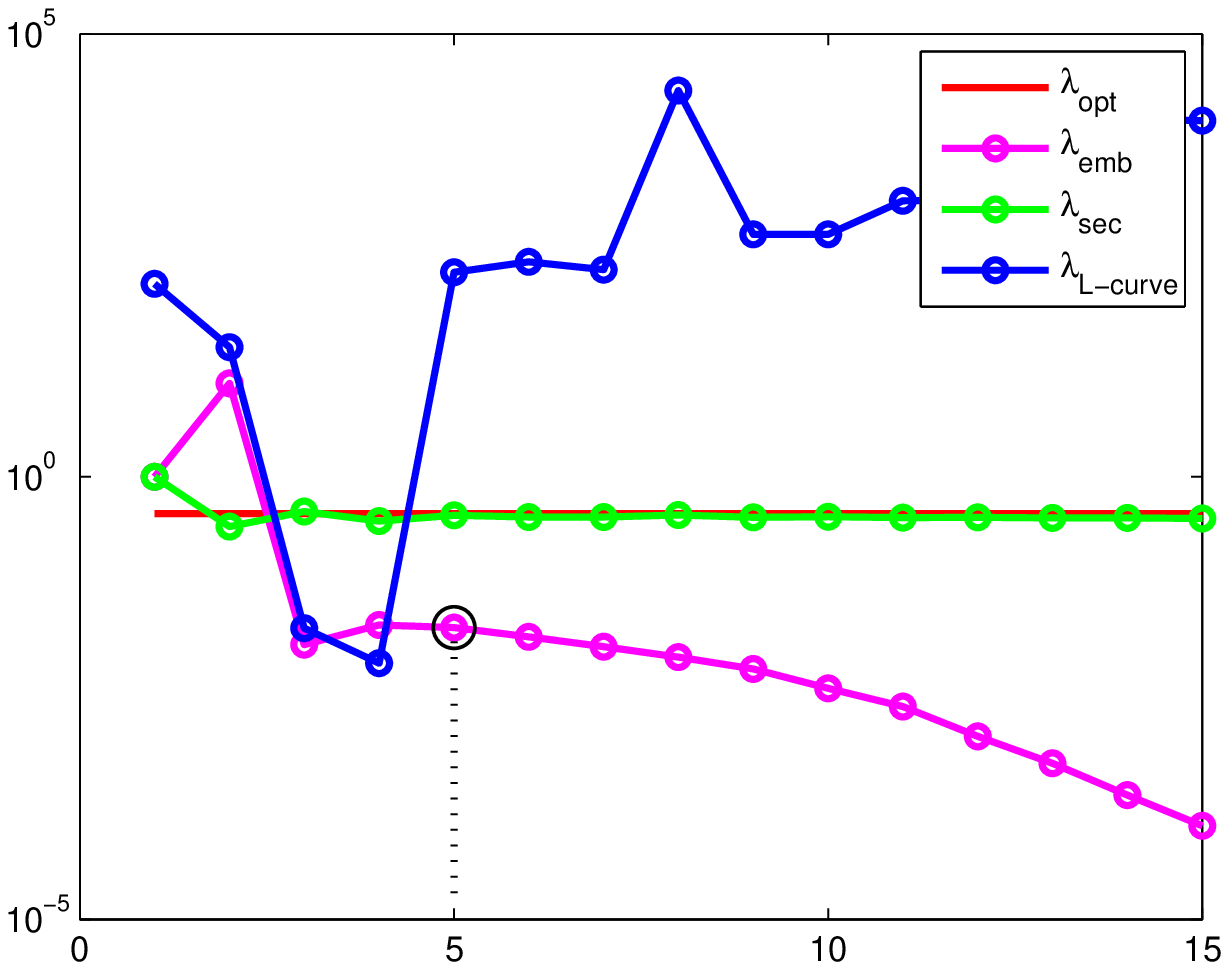} \newline
\par
\includegraphics[width=0.32\textwidth]{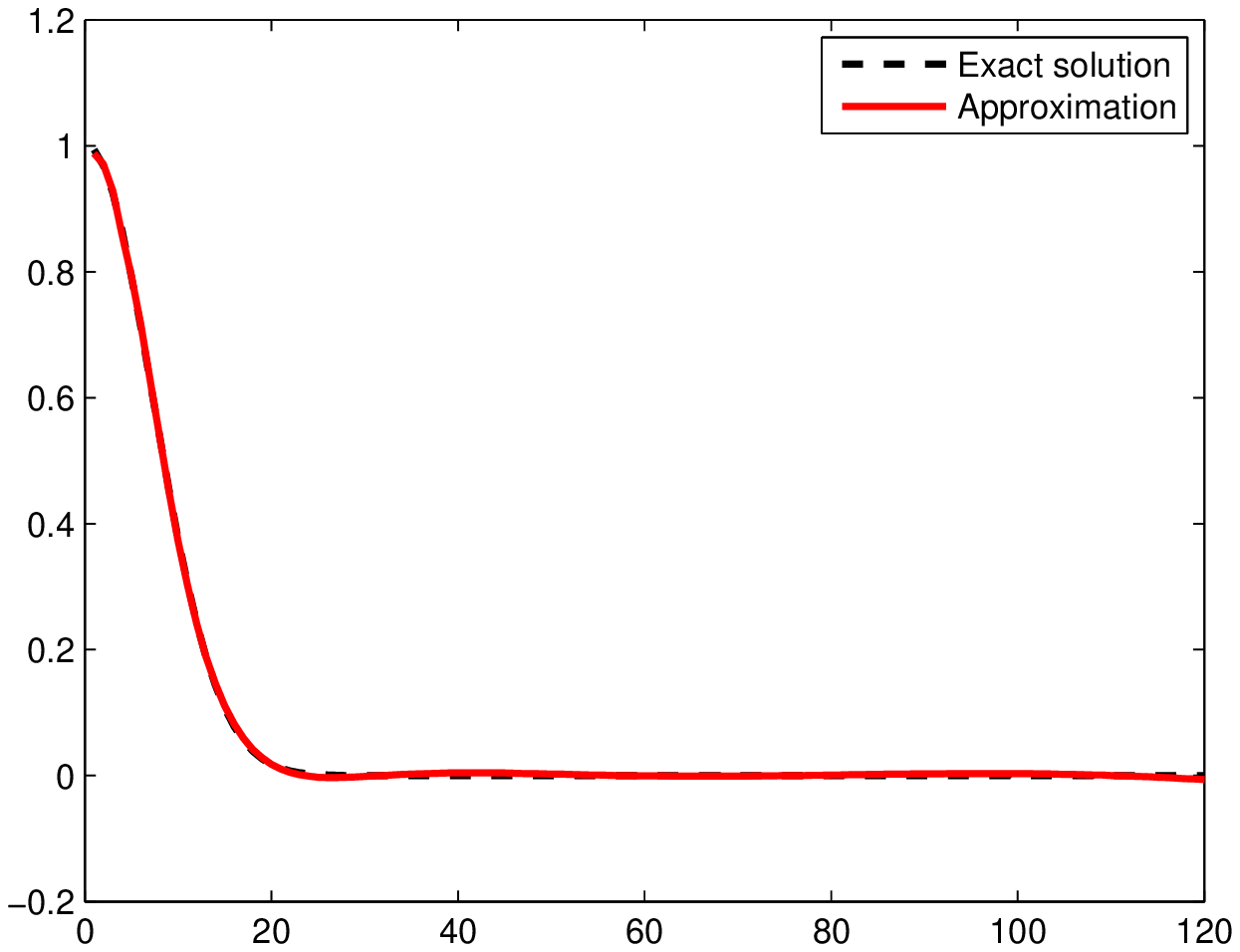} %
\includegraphics[width=0.32\textwidth]{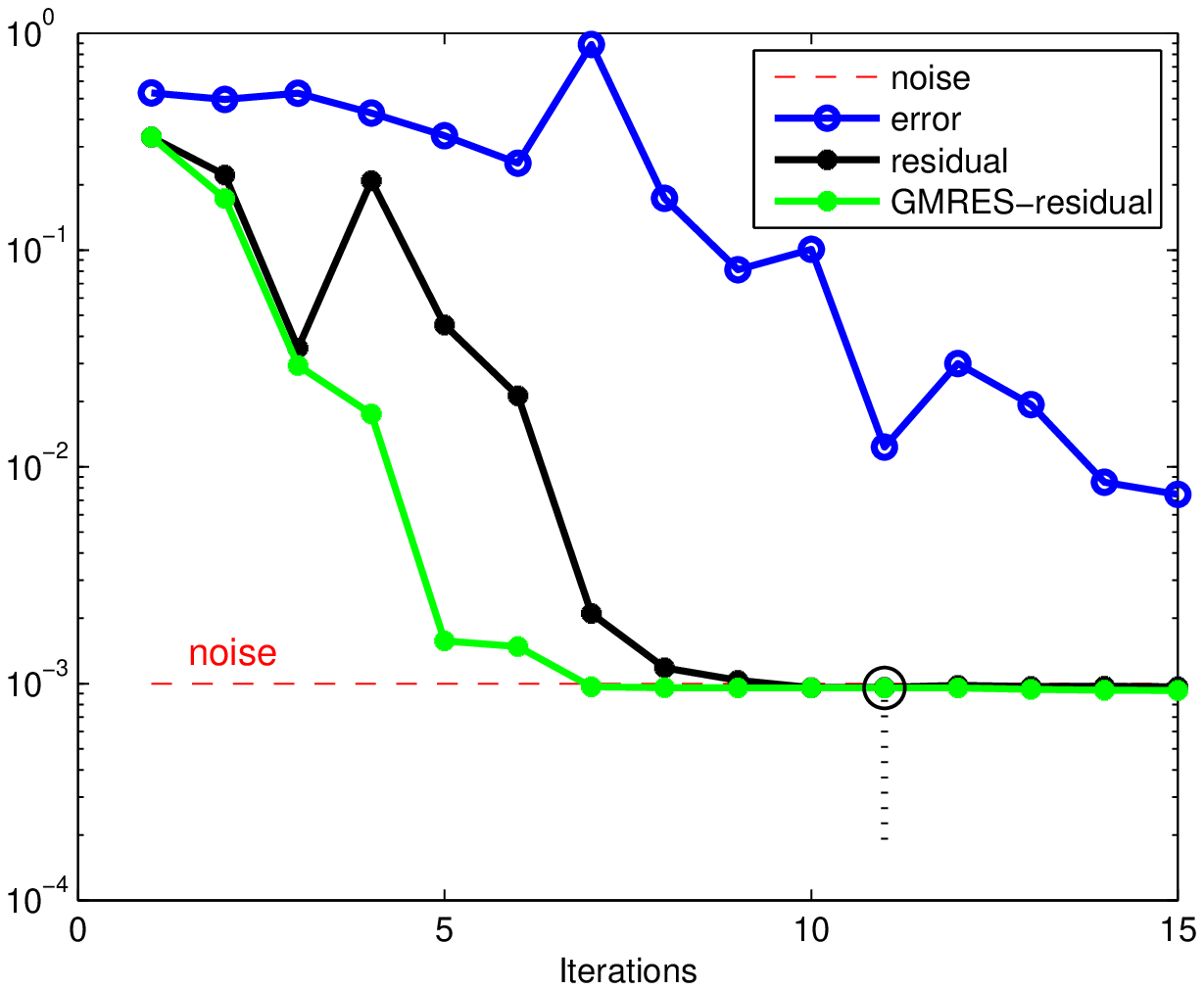} %
\includegraphics[width=0.32\textwidth]{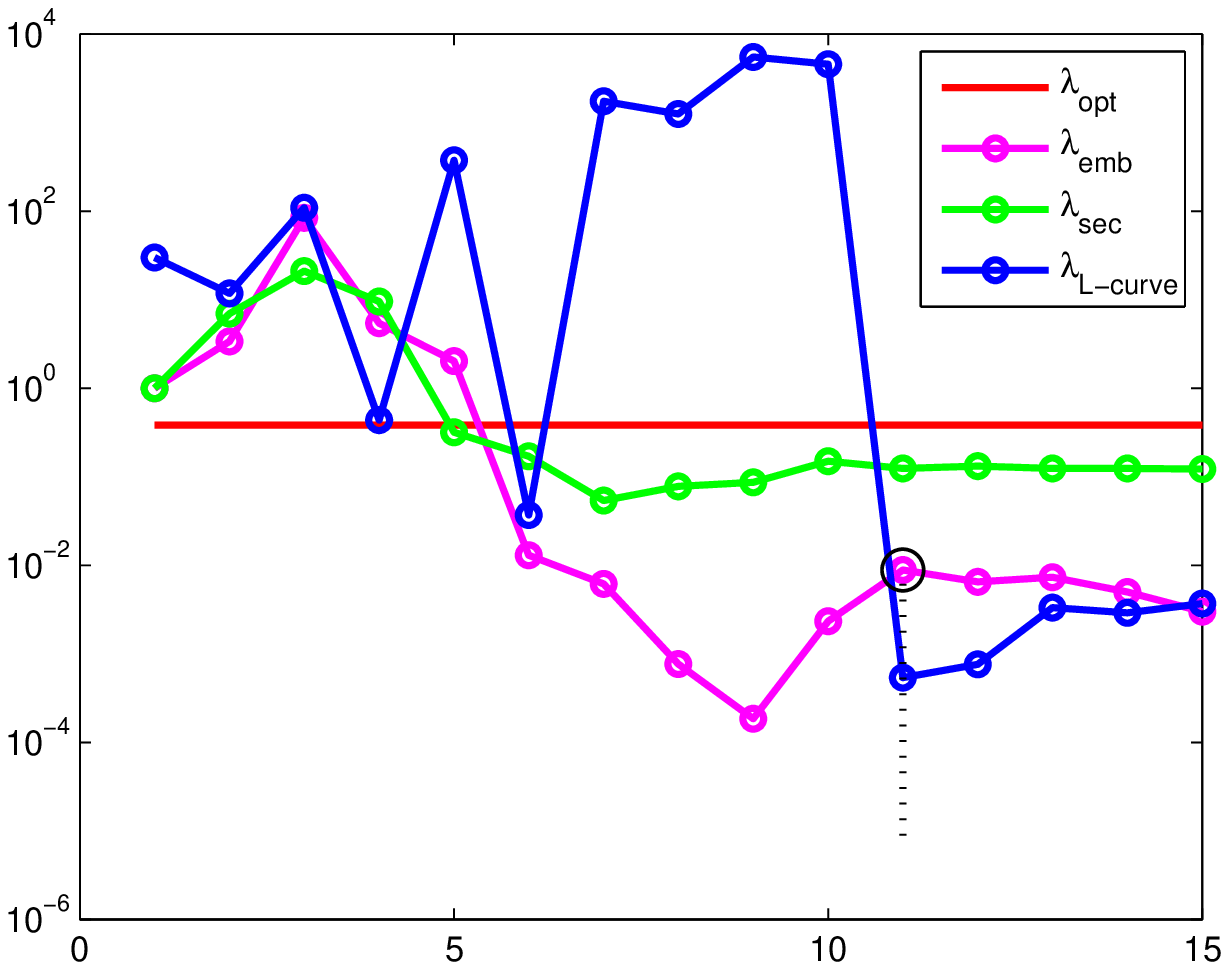} \newline
\par
\includegraphics[width=0.32\textwidth]{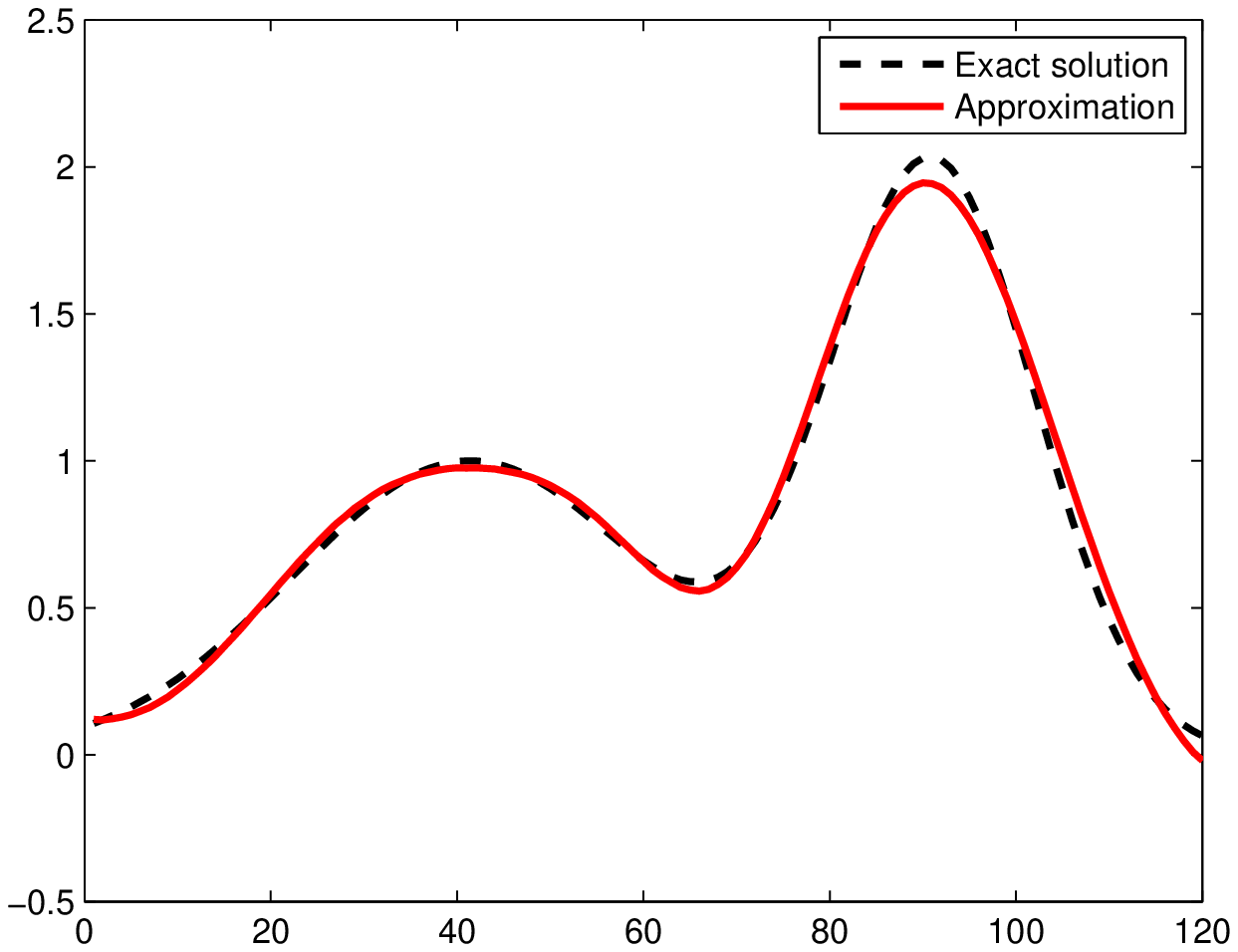} %
\includegraphics[width=0.32\textwidth]{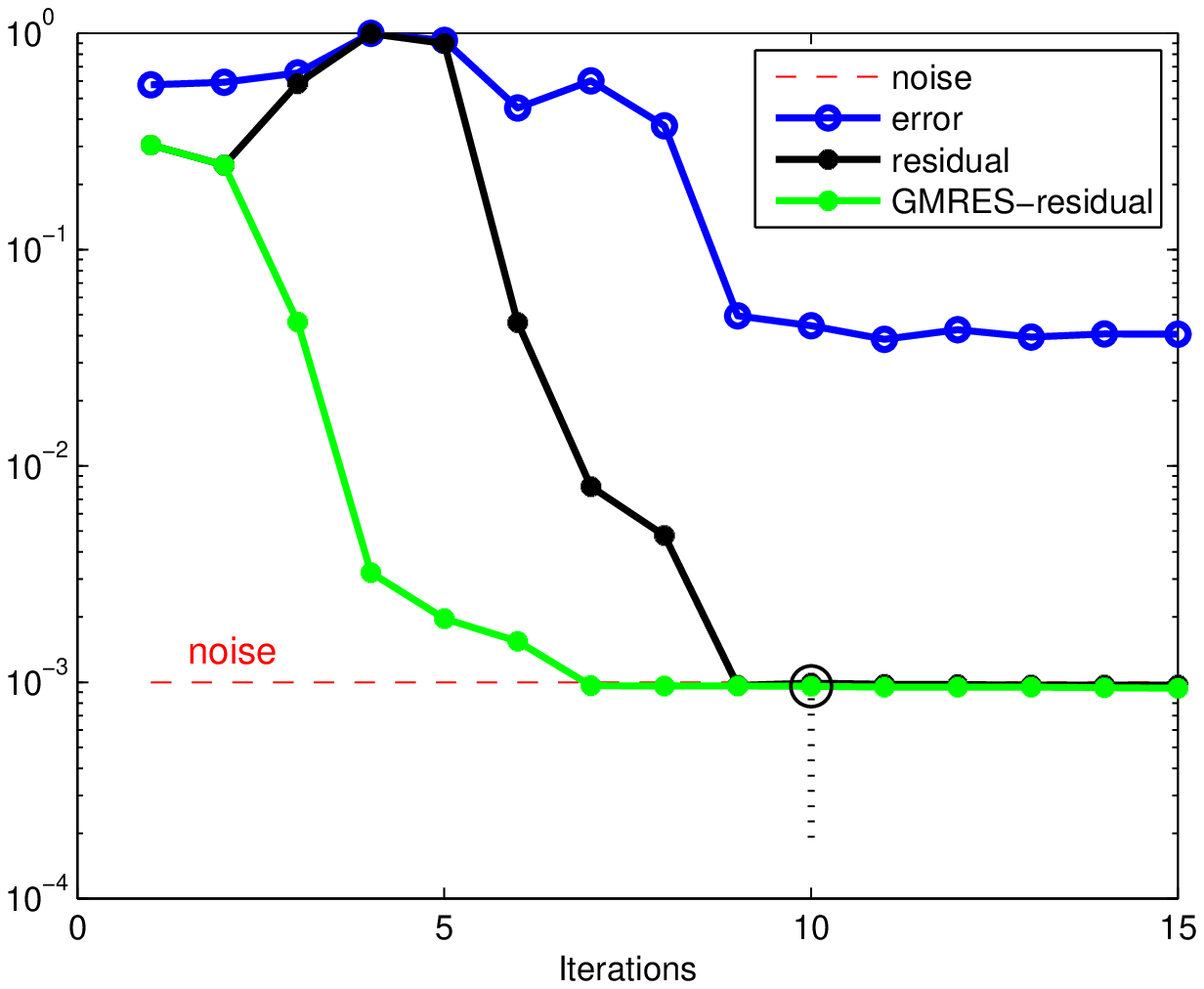} %
\includegraphics[width=0.32\textwidth]{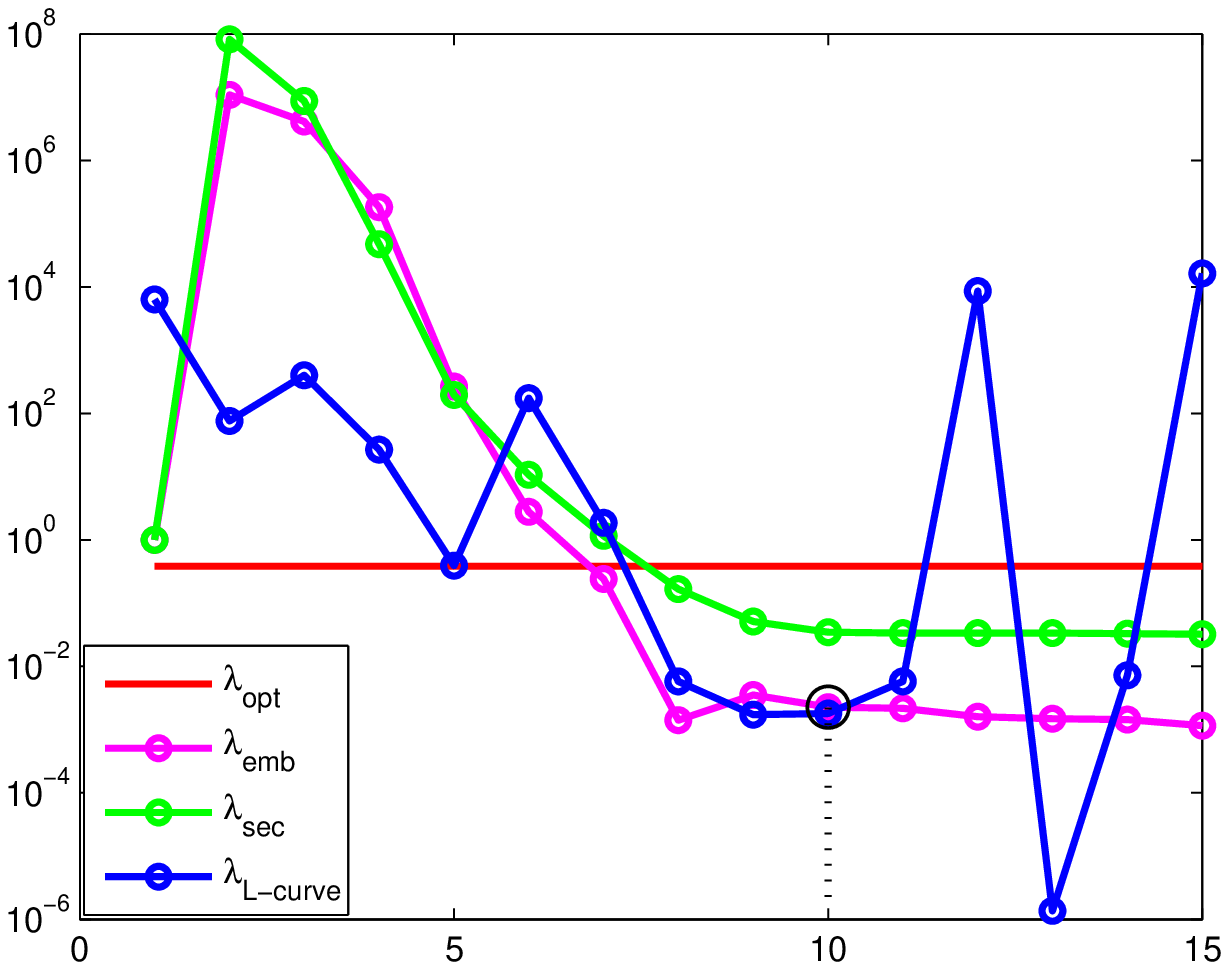}
\caption{From top to bottom: results for \texttt{baart}, \texttt{foxgood},
\texttt{i\_laplace}, \texttt{shaw}. On the left column we display the
computed approximate solution. In the middle column we show the convergence
behavior of the new method (error, discrepancy and GMRES residual) with the
noise level highlighted by a dashed lines. On the right we compare different
parameter choice strategies. The tick circle displayed in all the
frame of the middle and the rightmost columns marks the iteration at which we
would stop, according to the rule (\protect\ref{stopRES}), (\protect\ref%
{stopDISCR}). The approximate solutions refer to this iteration.}
\label{F1}
\end{figure}

For each experiment we show: a) the approximate solution; b) the relative
residual and error history; c) the value of the regularization parameter
computed at each iteration by the secant update method ($\lambda _{\mathrm{%
sec}}$) given by formula (\ref{lambdaABS}), the embedded method ($\lambda _{%
\mathrm{emb}}$) computed by (\ref{NEWupdate}), the ones arising from the
L-curve criterion ($\lambda _{\mathrm{L-curve}}$) see \cite{ATfirst}, and
the optimal one ($\lambda _{\mathrm{opt}}$) for the original,
full-dimensional regularized problem (\ref{GenTikh}) obtained by the
minimization of the distance between the regularized and the exact solution
\cite{OL1}
\begin{equation*}
\min_{\lambda }\left\Vert x_{\lambda}-x^{ex}\right\Vert ^{2}= \min_{\lambda
}\left\Vert \sum_{i=1}^{P} \frac{\lambda ^{2}}{(\gamma _{i}^{2}+\lambda ^{2})%
}\frac{\bar{u}_{i}^{T}b}{\sigma _{i}}x_{i}+\sum_{i=P+1}^{N}(u_{i}^{T}b)x_{i}
-\sum_{i=1}^{N}\frac{u_{i}^{T}b^{ex}}{\sigma _{i}}v_{i} \right\Vert,
\end{equation*}%
where $\gamma_i$, $\bar{u}_i$, $i=1,\dots,P$ are respectively the
generalized singular values and left generalized singular vectors of $(A,L)$%
, and $x_i$, $i=1,\dots,N$ are the right generalized singular vectors of $%
(A,L)$.

\subsection{Results for Image Restoration}

To test the performance of our algorithm in the image restoration contest, a
number of experiments were carried out, some of which are presented here.

Let $X$ be a $n\times n$ two dimensional image. The vector $x^{ex}$ of
dimension $N=n^{2}$ obtained by stacking the columns of the image $X$ and
the associated blurred and noise-free image $b^{ex}$ is generated by
multiplying $x^{ex}$ by a blurring matrix $A\in \mathbb{R}^{N\times N}$. The
matrix $A$ is block Toeplitz with Toeplitz blocks and is implemented in the
function \texttt{blur} from \cite{H1}, which has two parameters, \texttt{band%
} and \texttt{sigma}; the former specifies the half-bandwidth of the
Toeplitz blocks and the latter the variance of the Gaussian point spread
function. We generate a blurred and noisy image $b\in \mathbb{R}^{N}$ by
adding a noise-vector $e\in \mathbb{R}^{N}$, so that $b=Ax^{ex}+e$. We
assume the blurring operator $A$ and the corrupted image $b$ to be available
while no information is given on the error $e$.

In the example, the original image is the \texttt{cameraman.tif} test image
from Matlab, a $256\times 256$, 8-bit gray-scale image, commonly used in
image deblurring experiments. The image is blurred with parameters \texttt{%
band}=7 and \texttt{sigma}=2. We further corrupt the blurred images with
0.1\% additive Gaussian noise. The blurred and noisy image is shown in the
center column of Figure \ref{F2}, the regularization operator is defined as
\begin{equation}  \label{ImRegOp}
L=I_{n}\otimes L_{1}+L_{1}\otimes I_{n}\in \mathbb{R}^{N\times N},
\end{equation}%
(cf. \cite[\S 5]{KHE}). The restored image is shown in the right column of
Figure \ref{F2} . The result has been obtained in $m=8$ iterations of the
Arnoldi algorithm, the CPU-time required for this experiment is around $1.2$
seconds. Many other experiments on image restoration have shown similar
performances.

\begin{figure}[]
\centering
\includegraphics[width=0.95\textwidth]{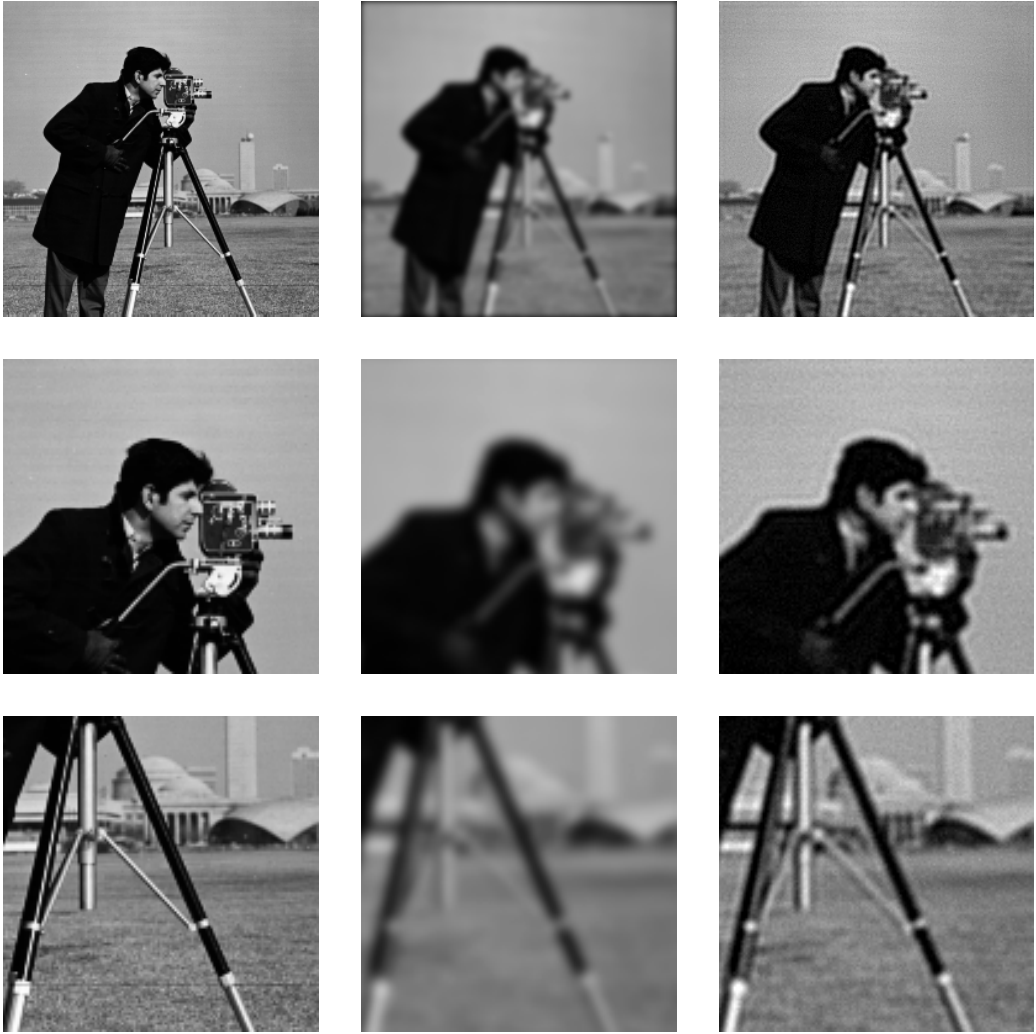}
\caption{Restoration of \texttt{cameraman.tif}. From left to right: original
image; blurred and noisy image with blur parameters \texttt{band}=7, \texttt{%
sigma}=2 and noise level $\protect\varepsilon =10^{-3} $; restored image.
From top to bottom: original-size image and two zooms.}
\label{F2}
\end{figure}

\subsection{Results for MRI Reconstruction}

The treatment of different kinds of medical images such as Magnetic
Resonance Imaging (MRI), Computed Tomography (CT), Position Emission
Tomography (PET), often requires the usage of image processing techniques to
remove various types of degradations such as noise, blur and contrast
imperfections. Our experiments focus on MRI medical image affected by
Gaussian blur and noise. Typically, when blur and noise affect the MRI
images, the visibility of small components in the image decreases and
therefore image deblurring techniques are extensively employed
to grant the image a sharper appearance.

In our test we blur a synthetic MRI $256\times 256$ image, with Gaussian
blur (\texttt{band}=9, \texttt{sigma}=2.5), and we add 10\% Gaussian white
noise, since the noise level of a real problem may be expected to be quite
high.

\begin{figure}[]
\centering
\includegraphics[width=0.95\textwidth]{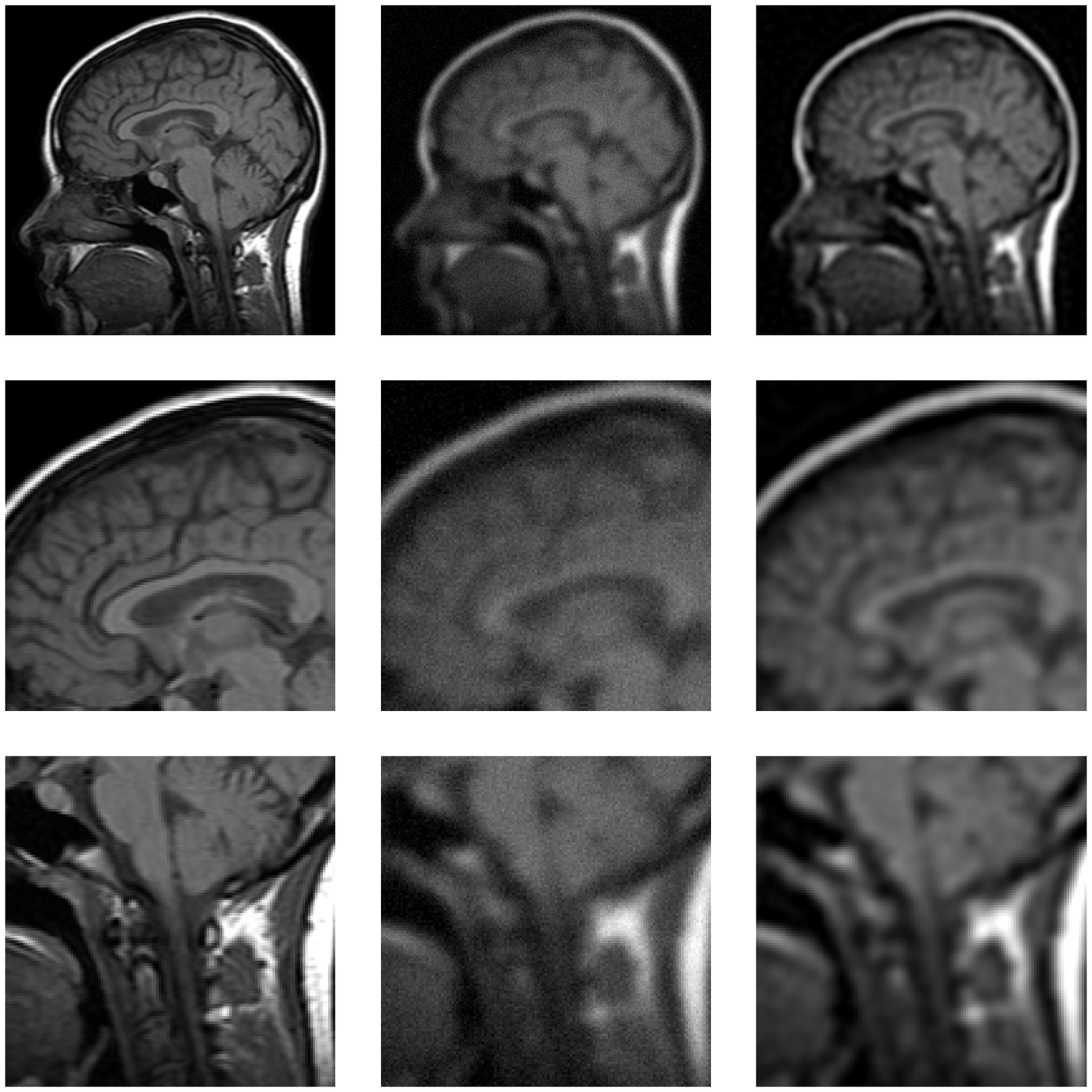}
\caption{Restoration of the test image \texttt{mri.png} image. From left to
right: original image; blurred and noisy image with noise level $\protect%
\varepsilon =10^{-1} $ and blur parameters \texttt{band}=9, \texttt{sigma}%
=2.5; restored image. From top to bottom original size image and two zooms.}
\label{F3}
\end{figure}

Figure \ref{F3} displays the performance of the algorithm. On the left
column we show the blur-free and noise-free image, on the middle column we
show the corrupted image, on the right column we show the restored image.The
regularization operator employed is again (\ref{ImRegOp}). The result has
been obtained in $m=5$ iterations of the algorithm, in around $0.7$ seconds.

\section{Conclusions}

In this paper we have proposed a very simple method to define the sequence
of regularization parameters for the Arnoldi Tikhonov method, in absence of
information on the percentage of error which affects the right hand side.
The numerical results have shown that this technique is rather stable, with
results comparable with the existing approaches (GCV, L-curve). We have used
the term "embedded" to describe this procedure since the construction of the
Krylov subspaces is used, at the same time, either as error estimator by means of
the GMRES residual or for the solution of (\ref{GenTikh}) with the AT method. We
remark that, in principle, the idea can be applied to any basic iterative
method able to approximate $\left\Vert e\right\Vert $ and, at the same time,
usable in connection with Tikhonov regularization (as for instance,
probably, the Lanczos bidiagonalization).

\section{Appendix}

While, in general, the SVD decomposition can be considered independent of
the Arnoldi process in absence of hypothesis on the starting vector $b$, the
following proposition states that, if the Discrete Picard Condition is
satisfied, then we are able to express a relation between $R(U_{m})$ (the
space generated by the columns of $U_{m}$, where $U_{m}\Sigma
_{m}V_{m}^{T}$ is the truncated SVD of $A$) and $\mathcal{K}_{m}(A,b)$. In
order to reduce the complexity of the notations, with respect to Section 4
here $b$ simply denotes the unperturbed right-hand side of the system.

\begin{proposition}
Assume that the singular values $A$ are of the type $\sigma
_{j}=O(e^{-\alpha j})$ ($\alpha >0$). Assume moreover that the Discrete
Picard Condition is satisfied. Let $\widetilde{V}_{m}:=\left[ \widetilde{v}%
_{0},...,\widetilde{v}_{m-1}\right] \in \mathbb{R}^{N\times m}$ where $%
\widetilde{v}_{k}:=A^{k}b/\left\Vert A^{k}b\right\Vert $. If $\widetilde{V}%
_{m}$ has full column rank, then there exist $C_{m}\in \mathbb{R}^{m\times
m} $ nonsingular, $E_{m},F_{m}\in \mathbb{R}^{N\times m}$, such that%
\begin{eqnarray}
\widetilde{V}_{m} &=&U_{m}C_{m}+E_{m},\quad \left\Vert E_{m}\right\Vert
=O(m^{1/2}\sigma _{m}),  \label{rr} \\
U_{m} &=&\widetilde{V}_{m}C_{m}^{-1}+F_{m},\quad \left\Vert F_{m}\Sigma
_{m}\right\Vert =O(m^{3/2}\sigma _{m}).  \label{rr2}
\end{eqnarray}
\end{proposition}

\begin{proof}
Let $U_{m}^{\perp }:=\left[ u_{m+1},...,u_{N}\right] \in \mathbb{R}^{N\times
(N-m)}$. Defining $C_{m}:=U_{m}^{T}\widetilde{V}_{m}\in \mathbb{R}^{m\times
m}$ and $E_{m}:=U_{m}^{\perp }\left( U_{m}^{\perp }\right) ^{T}\widetilde{V}%
_{m}\in \mathbb{R}^{N\times m}$ we have $\widetilde{V}_{m}=U_{m}C_{m}+E_{m}$%
. Now we observe that for $0\leq k\leq m-1$%
\begin{equation}
\left\vert u_{j}^{T}\widetilde{v}_{k}\right\vert \sim \sigma _{j}.
\label{rd}
\end{equation}%
For $k=0$ the above relation is ensured by the Picard Condition, whereas for
$k\geq 1$ it holds since%
\begin{equation*}
\widetilde{v}_{k}=\frac{\left\Vert A^{k-1}b\right\Vert }{\left\Vert
A^{k}b\right\Vert }A\widetilde{v}_{k-1}.
\end{equation*}%
Therefore, using $\sigma _{j}=O(e^{-\alpha j})$, we immediately obtain%
\begin{equation}
\left\Vert E_{m}\right\Vert =\left\Vert \left( U_{m}^{\perp }\right) ^{T}%
\widetilde{V}_{m}\right\Vert =O(m^{1/2}\sigma _{m}),  \label{m1}
\end{equation}

We observe that the matrix $C_{m}$ can be written as%
\begin{equation*}
C_{m}=U_{m}^{T}W_{m}S_{m},
\end{equation*}%
where $S_{m}$ is upper triangular and nonsingular if $\widetilde{V}_{m}$ has
full rank. Now, from the relation \cite[\S 2.6.3]{Golub}%
\begin{equation*}
\sigma _{\min }(U_{m}^{T}W_{m})^{2}=1-\left\Vert \left( U_{m}^{\perp
}\right) ^{T}W_{m}\right\Vert ^{2},
\end{equation*}%
the quantity $\left\Vert \left( U_{m}^{\perp }\right) ^{T}W_{m}\right\Vert $%
, which express the distance between $R(U_{m})$ and $R(W_{m})$, is strictly
less than one if the Picard Condition is satisfied. Thus, by (\ref{rr}), we
can write%
\begin{equation}
U_{m}=\widetilde{V}_{m}C_{m}^{-1}-E_{m}C_{m}^{-1},  \label{um}
\end{equation}%
and since $E_{m}=U_{m}^{\perp }\left( U_{m}^{\perp }\right) ^{T}\widetilde{V}%
_{m}$ we have that
\begin{equation}
E_{m}C_{m}^{-1}=U_{m}^{\perp }\left( U_{m}^{\perp }\right) ^{T}\widetilde{V}%
_{m}\left( U_{m}^{T}\widetilde{V}_{m}\right) ^{-1}.  \label{em}
\end{equation}%
By (\ref{rd}), using the Cramer rule to compute $\left( U_{m}^{T}\widetilde{V%
}_{m}\right) ^{-1}\Sigma _{m}\in \mathbb{R}^{m\times m}$ we can see that
each element of this matrix is of the type $O(1)$, so that%
\begin{equation*}
\left\vert \left( U_{m}^{\perp }\right) ^{T}\widetilde{V}_{m}\left( U_{m}^{T}%
\widetilde{V}_{m}\right) ^{-1}\Sigma _{m}\right\vert \sim m\left(
\begin{array}{ccc}
\sigma _{m+1} & \cdots  & \sigma _{m+1} \\
\vdots  &  & \vdots  \\
\sigma _{N} & \cdots  & \sigma _{N}%
\end{array}%
\right) \in \mathbb{R}^{(N-m)\times m},
\end{equation*}%
and hence%
\begin{equation}
\left\Vert \left( U_{m}^{\perp }\right) ^{T}\widetilde{V}_{m}\left( U_{m}^{T}%
\widetilde{V}_{m}\right) ^{-1}\Sigma _{m}\right\Vert =O(m^{3/2}\sigma _{m}),
\label{mm}
\end{equation}%
using again $\sigma _{j}=O(e^{-\alpha j})$. Defining $F_{m}=-E_{m}C_{m}^{-1}$
we obtain (\ref{rr2}) by (\ref{um}), (\ref{em}) and (\ref{mm}).
\end{proof}

Thanks to the above Proposition, the following proof of Theorem \ref{th1}
stated in Section 4.1 is straightforward.

\noindent \textbf{Proof of Theorem }\ref{th1}. Let $A_{m}=U_{m}\Sigma
_{m}V_{m}^{T}$, and let $\Delta _{m}=A-A_{m}$.
By (\ref{Arnoldi})%
\begin{eqnarray*}
h_{m+1,m} &=&w_{m+1}^{T}Aw_{m} \\
&=&w_{m+1}^{T}\Delta _{m}w_{m}+w_{m+1}^{T}A_{m}w_{m} \\
&=&O(\sigma _{m+1})+w_{m+1}^{T}U_{m}\Sigma _{m}V_{m}^{T}w_{m},
\end{eqnarray*}%
since $\left\Vert \Delta _{m}\right\Vert =\sigma _{m+1}$. Therefore, using (%
\ref{rr2}) we obtain%
\begin{equation*}
h_{m+1,m}=O(\sigma _{m+1})+w_{m+1}^{T}(\widetilde{V}_{m}C_{m}^{-1}+F_{m})%
\Sigma _{m}V_{m}^{T}w_{m}.
\end{equation*}%
which concludes the proof, since $w_{m+1}^{T}\widetilde{V}_{m}=0$ and $%
\left\Vert F_{m}\Sigma _{m}\right\Vert =O(m^{3/2}\sigma _{m})$. \hspace{2cm}
$\square $

\begin{remark}
The hypothesis $\sigma _{j}=O(e^{-\alpha j})$ apparently limits the above
results to severely ill-conditioned problems. Actually, it is just used in (%
\ref{m1}) and (\ref{mm}) since, by the integral criterion,%
\begin{equation*}
\sum\nolimits_{j\geq m+1}\sigma _{j}=O(e^{-\alpha m})=O(\sigma _{m}).
\end{equation*}%
In this sense, the results can be extended to mildly ill-conditioned
problems, in which $\sigma _{j}=O(j^{-\alpha })$, $\alpha >1$. In this
situation we would have%
\begin{equation*}
\sum\nolimits_{j\geq m+1}\sigma _{j}=O(m^{1-\alpha }),
\end{equation*}%
so that, for $\alpha $ sufficiently large, (\ref{rr}), (\ref{rr2}) and the
results of Theorem \ref{th1} and Corollary \ref{cor1}, can be extended to
mildly ill-conditioned problems by replacing $\sigma _{m}$ with $%
O(m^{1-\alpha })$.
\end{remark}

\end{document}